\documentclass[11pt,reqno,tbtags]{amsart}

\pdfoutput=1 

\usepackage{amssymb}
\usepackage{upgreek}
\usepackage{graphicx}
\usepackage{natbib}
\numberwithin{equation}{section}
\usepackage{vmargin}
\setmargrb{20mm}{20mm}{20mm}{20mm}
\usepackage[parfill]{parskip}

\makeatletter
\newtheorem*{rep@theorem}{\rep@title}
\newcommand{\newreptheorem}[2]{%
\newenvironment{rep#1}[1]{%
 \def\rep@title{#2 \ref{##1}}%
 \begin{rep@theorem}}%
 {\end{rep@theorem}}}
\makeatother

\newtheorem{theorem}{Theorem}[section]
\newreptheorem{theorem}{Theorem}
\newtheorem{lemma}[theorem]{Lemma}

\theoremstyle{definition}

\theoremstyle{remark}

\newcounter{thmenumerate}

\newcounter{xenumerate}

\newcommand\E{\operatorname{\mathbb E{}}}

\renewcommand\Pr{\operatorname{\mathbb P{}}}

\newcommand\Var{\operatorname{Var}}

\newcommand\Exp{\operatorname{Exp}}

\newcommand\eps{\varepsilon}

\renewcommand\phi{\varphi}

\newcommand\cR{\mathcal R}

\newcommand\Z{{\mathbb Z}}

\newcommand\R{{\mathbb R}}

\newcommand{\myvec}[1]{ \mathbf{#1} }
\newcommand{\mymat}[1]{ \boldsymbol{#1} }

\newcommand{\mygvec}[1]{ \boldsymbol{#1} }
\newcommand{\be}{\begin{equation}}
\newcommand{\ee}{\end{equation}}
\newcommand{\ba}{\begin{equation} \begin{aligned}}
\newcommand{\ea}{\end{aligned} \end{equation}}
\newcommand{\ddt}[1]{\frac{\mathrm{d}#1}{\mathrm{d}t}}

\begin{document}
\title[Extinction Times in the Stochastic Logistic Epidemic]{Extinction Times in the Subcritical Stochastic SIS Logistic Epidemic}

\date{\today} 

\author{Graham Brightwell}
\address{Department of Mathematics, London School of Economics,
Houghton Street, London WC2A 2AE, United Kingdom}
\email{g.r.brightwell@lse.ac.uk}
\urladdr{http://www.maths.lse.ac.uk/Personal/graham/}

\author{Thomas House} \thanks{Thomas House is supported by the EPSRC, grant
reference EP/N033701/1.}
\address{School of Mathematics, University of Manchester}
\email{thomas.house@manchester.ac.uk}
\urladdr{http://personalpages.manchester.ac.uk/staff/thomas.house/}

\author{Malwina Luczak} \thanks{The research of Malwina Luczak was largely carried out at Queen Mary, University of London, and was supported by
an EPSRC Leadership Fellowship, grant reference EP/J004022/2.}
\address{School of Mathematics and Statistics, University of Melbourne, Parkville, Victoria, Australia}
\email{mluczak@unimelb.edu.au}


\begin{abstract}
Many real epidemics of an infectious disease are not straightforwardly super-
or sub-critical, and the understanding of epidemic models that exhibit such
complexity has been identified as a priority for theoretical work. We provide
insights into the near-critical regime by considering the stochastic SIS
logistic epidemic, a well-known birth-and-death chain used to model
the spread of an epidemic within a population of a given size~$N$.\\

\noindent{}We study
the behaviour of the process as the population size $N$ tends to infinity.  Our
results cover the entire subcritical regime, including the ``barely subcritical'' regime,
where the recovery rate exceeds the infection rate by an amount that tends to~0
as $N \to \infty$ but more slowly than $N^{-1/2}$.  We derive precise
asymptotics for the distribution of the extinction time and the total number of
cases throughout the subcritical regime, give a detailed description of the
course of the epidemic, and compare to numerical results for a range of
parameter values.  We hypothesise that features of the course of the epidemic
will be seen in a wide class of other epidemic models, and we use real data to
provide some tentative and preliminary support for this theory.
\end{abstract}

\maketitle

\section{Introduction}\label{Sintro}

\subsection{Epidemiological motivation}

Models of the dynamics of disease spread are widely used throughout infectious disease
epidemiology, and inform an increasing number of health policy
domains (Heesterbeek et al. 2015).  Typically, epidemic models have a quantity
called the basic reproduction ratio $\mathcal{R}_0$ such that, if $\cR_0 > 1$,
then the epidemic is supercritical and grows exponentially and, if $\cR_0 < 1$, then the
epidemic is subcritical and shrinks exponentially; see for instance Diekmann, Heesterbeek and Britton (2013).

Increasingly, however, epidemiologists are confronted with epidemics where the dynamics cross the threshold from supercritical to
subcritical (see for instance Klepac et al (2013) and other papers in that journal issue) due to control measures, or from subcritical to
supercritical due, for example, to mutation (Antia et al.\ 2003; Bull and Dykhuizen 2003; Scheffer et al.\ 2009; O'Regan and Drake 2013), or
where the behaviour is not exponential (Chowell et al.\ 2016).  This has led to the
understanding of near-critical epidemics being highlighted as a key challenge for
disease-dynamic models of infectious diseases (Britton et al.\ 2015).

In this paper, we are interested in understanding the course of subcritical epidemics, especially
where $\cR_0$ is close to~1.  We shall give a detailed analytical study of a particularly simple epidemic process,
the \emph{stochastic SIS logistic process} (also called the SIS model, the
contact process, or the logistic model) in the subcritical regime.  In this model, each member of a population of fixed size $N$
is either susceptible or infective.  Infective individuals encounter a random other member of the population at rate $\lambda$, and
infect them if they are susceptible: infective individuals recover at rate $\mu$, and once recovered they are immediately susceptible
again.  The state of the epidemic is effectively determined by the number $X_N(t)$ of infectives,
or by the {\em prevalence} $X_N(t)/N$, at time~$t$.  We defer the formal definitions to the next section.
For most purposes, this process is too simple to reflect the full behaviour of a real-world epidemic, but it can be suitable
for modelling sexually transmitted and hospital-acquired infections (Eames and Keeling 2002; Ross and Taimre 2007).

In our model, the basic reproduction ratio $\cR_0$ is equal to $\lambda/\mu$, and we are especially interested in the regime where
$\cR_0$ tends to~1 from below as $N$ tends to infinity.  When also $(1-\cR_0) N^{1/2} \to \infty$, we are in the
{\em barely subcritical regime}.  We note the following features of the typical course of the epidemic in this regime, all of which
distinguish regimes near criticality from those where $\cR_0$ is fixed and less than~1.  We make the
hypothesis that these features are common to a wide class of epidemic models (not only SIS models) in barely subcritical regimes.
We express our statements in terms of the basic reproduction ratio $\cR_0$, the population size $N$, and a speed parameter $\mu$,
where $1/\mu$ is roughly the expected duration of an individual case.
\begin{itemize}
\item The time to extinction is of order much larger than $\log N /\mu$ (typically it is of order $\log N /\mu(1-\cR_0)$).
\item There is a period of time before extinction, of order $1/\mu(1-\cR_0)$, where the number of infectives follows a track resembling a
random walk, remaining of order at most $1/(1-\cR_0)$ throughout.  The duration of this period is not well-concentrated.
\end{itemize}

The {\em cut-off phenomenon}, in the case of a stochastic epidemic process, is that the typical time to extinction is much greater than the window of time
over which the probability of extinction goes from near~0 (at this stage, the number of infectives will typically be larger than $1/(1-\cR_0)$, but
much smaller than the population size) to near~1.  The SIS logistic process exhibits this phenomenon in the subcritical regime,
with the expected time to extinction around $\log [N (1-\cR_0)^2] / \mu (1-\cR_0)$, and the window where extinction typically occurs having width of order
$1/\mu(1-\cR_0)$: thus cut-off becomes less pronounced as we approach the {\em critical regime} where $|1-\cR_0|$ is of order at most $N^{-1/2}$.
We hypothesise that this weakening of cut-off is also a feature of many barely subcritical epidemic processes.

Figure~\ref{fig:data} shows three real examples of the courses of epidemics after they became subcritical due to control efforts.
These concern: the recent Ebola epidemic in Sierra Leone (top plot), smallpox (middle plot),
and polio (bottom plot).  For all three examples, we superimpose curves showing smooth decay (in the top plot the curves are solutions of an ODE:
in the others, they are exponential decay curves at various rates $r$), illustrating that these provide a poor fit for the observed detailed dynamics of
the disease over time, especially in the end stages of the epidemic (see in particular the right hand panels in the bottom two plots, which are rescaled versions of the plots showing the behaviour more clearly).  However, at a glance, the plots do exhibit behaviour very broadly in line with
our expectations for a barely subcritical epidemic.  We discuss the data in more detail in Appendix~A.

For comparison, Figure~\ref{fig:mc} shows the behaviour of the stochastic SIS logistic process
for a population of size $10^6$ with $10^3$ initial cases as $\lambda$ is
increased towards $\mu=1$ from below.  These sample paths are compared to ``smooth decay'' curves of the form
$e^{(\lambda-\mu)t}$; the figure illustrates our analytic results for the model: individual realisations of the stochastic
SIS logistic process deviate significantly from smooth exponential decay in the final stages, and exhibit the behaviour we described above.

The potential implications of our work for more general epidemic models are discussed in more detail in Section~\ref{sec.mean-extinction}.

Our results on the stochastic SIS logistic process are the first to incorporate the
barely subcritical regime, where $\cR_0$ approaches $1$ from below as the population size tends to infinity.
We provide analytic methods for studying the SIS logistic process in the large-population limit,
and thereby obtain precise asymptotic results for
the distribution of extinction times and the total number of infection events.  We also give attention to numerical
approaches that are well adapted to the near-critical regime and allow the
exploration of model behaviour at a given value of the population size.
We expect that our methodology can be generalised to apply to other models of epidemiological and biological interest.

\subsection{Technical background and outline of paper}

The stochastic SIS logistic process is defined as follows.  Given a ``size
parameter'' $N$, and two further parameters $\lambda$ and $\mu$, let $X_N =
(X_N (t))_{t\ge 0}$ be the continuous-time Markov chain with state space $\{ 0,
\dots, N\}$, and transitions as follows:
\begin{eqnarray} \label{eq:eventsrates}
X &\to& X+1 \quad \mbox{ at rate } \quad \lambda X (1-X/N), \\
X &\to& X-1 \quad \mbox{ at rate } \quad \mu X. \nonumber	
\end{eqnarray}
This is the most basic stochastic model of the spread of an SIS
(susceptible-infective-susceptible) epidemic within a population of size $N$.
In this context, $X_N(t)$ represents the number of infective individuals at
time~$t$.  Each infective encounters a random other member of the population at
rate~$\lambda$; if the other individual is currently susceptible, they become
infective.  Also, each infective recovers at rate $\mu$; once they are
recovered they become susceptible again.  The stochastic SIS logistic process
is also used as a model for a metapopulation process, where $N$ represents the
number of available patches, $X_N(t)$ is the number of patches that are
populated at time~$t$, $\lambda$~represents the rate at which one existing
colony attempts to colonise another patch, and $\mu$ represents the rate at
which an entire patch becomes depopulated due to some catastrophe.  The model
was first formulated by Feller~(1939), and further studied by Bartlett~(1957).
It was rediscovered by Weiss and Dishon~(1971), and has since been investigated by many authors.  A recent thorough
treatment of the model is the book of N\aa sell~(2011), who in
particular mentions a number of other application areas and gives a large list
of further references.

Suppose to begin with that $\lambda$ and $\mu$ are fixed constants.  The key parameter is the ratio $\lambda/\mu$, which,
for a small-scale epidemic,
approximates the {\em basic reproduction ratio} $\cR_0$ of the epidemic, defined as the mean number of individuals infected by a single infective
in a large population of susceptibles.  The behaviour of the logistic process is
radically different depending on whether the quantity $\cR_0$ is greater or less than 1.  In the case $\cR_0 > 1$, the process typically takes a time
exponential in $N$ to die out, spending most of its duration near to the value $(\lambda - \mu) N / \lambda$ where the upward and downward
transition rates are equal: this models a {\em supercritical} epidemic, where an initially small number of infectives may generate an outbreak that
becomes endemic in the population for a very long period.  In the case $\cR_0 < 1$,
the process is always drifting downwards, and even an initially very large epidemic dies out with high probability in time of order $\log N$.
In this paper, we give precise results about the distribution of the extinction time in this {\em subcritical} case.
In the special case $\cR_0 = 1$, the expected extinction time is of order $\sqrt N$.

To study behaviour in the transition between the supercritical and subcritical
regimes, we regard $\lambda$ and $\mu$ as functions of the size parameter~$N$,
and we pay special attention to cases where $\cR_0 = \lambda/\mu$ tends to~1 as
$N$ tends to infinity.

Earlier work~(N\aa sell 1996; Dolgoarshinnykh and Lalley 2006; Kessler 2008) has demonstrated that there is a {\em critical window}
for the logistic process where $|\cR_0 - 1| = O(N^{-1/2})$; the change in the nature of the process
occurs as $\cR_0$ crosses this window.  Our main aim in this paper is to give the distribution of the extinction time throughout the entire subcritical
range, i.e., as long as $(1-\cR_0)N^{1/2} \to \infty$.  Our results show exactly how the extinction time changes from order $\log N$ to order
$\sqrt N$ as $\cR_0$ approaches~1 from below.

Our results for the case where $\lambda$ and $\mu$ are fixed real numbers with $\lambda < \mu$ are new,
but our main interest is in the case where $\lambda=\lambda(N)$ and $\mu=\mu(N)$ are functions of $N$ with
$\mu(N)-\lambda(N) \to 0^+$.  We shall always assume that $\mu(N)$ and $\lambda(N)$ are bounded away from both~0 and~$\infty$.
We choose to state all our results in terms of the two parameters $\lambda$ and $\mu$ for ease of comparison to earlier
results: however, apart from a constant factor determining the speed of the epidemic, all of our results could be restated in terms of $\cR_0$.
Note that, under our assumptions, $1 - \cR_0 = (\mu -\lambda)/\mu$ is of the same order as $\mu-\lambda$, and to say that our sequence of parameter
values is in the {\em subcritical regime} means that $(\mu-\lambda)N^{1/2} \to \infty$.

We also allow the initial state $X_N(0)$ to depend on~$N$.  One case of natural interest is where $X_N(0) \simeq \alpha N$ for some
$\alpha \in (0,1]$, but our results also cover the case where $X_N(0)/N \to 0$.
We set $T_N$ to be the time to extinction (i.e., the hitting time of the absorbing state~0) for $X_N(t)$, with infection rate $\lambda = \lambda (N)$, recovery rate $\mu = \mu(N)$, and initial state $X_N(0)$.  Our interest is in the asymptotic distribution of $T_N$,
as $N \to \infty$.

There is an exact expression for $\E T_N$ as a double summation, due to Weiss and Dishon~(1971) in the case where $X_N(0) = N$, and in general to
Leigh~(1981) and Norden~(1982).  The asymptotics of this sum have been determined in some cases, e.g., by Doering, Sargsyan and Sander (2005).
Our methods give precise information about the distribution of $T_N$, not just its expectation.

The logistic process is naturally associated with the differential equation
\begin{equation} \label{eq.diff-eq}
	\ddt{x} = \lambda x (1-x) - \mu x = \lambda x(1 - \mu/\lambda - x),
\end{equation}
where $x(t)$ represents the {\em proportion} of infective individuals at time $t$.  This equation was first studied by Verhulst~(1838), and it
is known as the {\em Verhulst equation} or {\em logistic equation}.  It follows from the general theory of Kurtz~(1971) that, as $N \to \infty$,
$X_N (t)/N$ is well concentrated around the solution $x(t)$ of the differential equation~(\ref{eq.diff-eq}), uniformly over fixed
time intervals, as long as $X_N(0)/N$ is well approximated by its initial condition $x(0)$.  For our purposes, we need to show concentration for
longer periods, and this is possible thanks to the special structure of the logistic equation when $\mu \ge \lambda$.

The behaviour of the deterministic process $x(t)$ also depends on whether $\cR_0$ is greater than, equal to, or less than~1.
In the case where $\lambda > \mu$ (i.e., $\cR_0 > 1$), there is a stable fixed point of the drift equation~(\ref{eq.diff-eq}) at $x = 1-\mu/\lambda$
(and an unstable fixed point at $x=0$).  If there are a large number of infective individuals at time~0, then with high
probability $X_N (t)/N$ heads rapidly towards the stable fixed point, then spends most of its time in the neighbourhood of that fixed point,
making excursions into the rest of the state space until eventually one of these excursions reaches the absorbing state~0.  Precise results are known about
the distribution of the time to extinction, which is exponential in $N$, and about the quasi-stationary
distribution, which is centred around the stable fixed point of~(\ref{eq.diff-eq}).  See, for instance: Barbour~(1976), Kryscio and Lef\`evre~(1989),
N\aa sell~(1996), Andersson and Djehiche~(1998) and the book of N\aa sell~(2011).

If $\lambda \le \mu$, then the differential equation~(\ref{eq.diff-eq}) has a single stable fixed point at~$x=0$, and all its solutions converge to zero
as $t\to \infty$.  For the corresponding Markov chain, it is also known that the epidemic dies out rapidly with high probability whenever
$\lambda$ and $\mu$ are fixed constants with $\lambda \le \mu$.

Doering, Sargsyan and Sander~(2005) give an asymptotic formula for the mean extinction time, in the case where $\lambda < \mu$ are fixed constants and
the initial state $X_N(0)$ is of order $N$:
\begin{equation} \label{dss}
\E T_N = \frac{1}{\mu-\lambda} (\log N + O(1)).
\end{equation}
In the case where $\lambda = \mu$, they obtain:
$$
\E T_N = \frac{1}{\lambda} \left[ \left( \frac{\pi}{2} \right)^{3/2} \sqrt N + \log N \right] + O(1).
$$
Doering, Sargsyan and Sander~(2005) also study the mean time to extinction starting from a state with a single infective individual, and
Kessler~(2008) extends these results to cover the whole of the ``transition region'', where $\mu-\lambda$ is of order $N^{-1/2}$.

A formula for the asymptotic distribution of the time $T_N$ to extinction, in the case where $\lambda < \mu$ and $X_N(0)/N$
tends to a constant, is presented by Kryscio and Lef\`evre~(1989) with a heuristic argument, and then reproduced by Andersson and Djehiche~(1998).
However, the formula is erroneous.  It was noted by Doering, Sargsyan and Sander~(2005) that the formula given by Kryscio and Lef\`evre~(1989) and
Andersson and Djehiche~(1998) is inconsistent with their result (\ref{dss}), and with their numerical results.
As far as we are aware, no correct explicit formula for the asymptotic distribution of the time $T_N$ to extinction when $\lambda < \mu$ has appeared
in the literature, even in the case where $\lambda$ and $\mu$ are fixed constants.  In his book,
N\aa sell~(2011) identifies two distinct regimes: one ``critical regime'', where $\mu-\lambda$ is of order at most $N^{-1/2}$, and another
(subcritical) where $\mu-\lambda$ is constant or tends to zero more slowly than $N^{-1/2}$.  For both regimes, N\aa sell~(2011) poses as an open
problem the determination of the mean extinction time $\E T_N$.  Our results have some similarity with Theorem~2(ii) of Sagitov and Shaimerdenova (2013), who study the distribution of the extinction time for a different version of the logistic model in a completely different limit.

Barbour, Hamza, Kaspi and Klebaner~(2015) study a
very general class of population models, which includes this one.  The distribution of the extinction time $T_N$, in the case where $\lambda$ and $\mu$ are
fixed with $\lambda < \mu$, can, with some effort, be derived from their Theorem~1.2.  Our results cover the case of fixed $\lambda$ and $\mu$, as well as the near-critical case (which Barbour et al (2015) do not cover), and our proof for this model
is significantly simpler than the general argument given by Barbour et al (2015).

Here we obtain the asymptotic distribution of $T_N$ throughout the subcritical regime, for general initial conditions.
Our main result is as follows.  We recall that a random variable $W$ has the {\em standard Gumbel distribution} if $\Pr(W \le w) = e^{-e^{-w}}$
for all $w \in \R$.  The mean of $W$ is equal to Euler's constant $\gamma \approx 0.5772$.

\begin{theorem} \label{thm.main}
Suppose that $\mu = \mu(N)$ and $\lambda = \lambda(N)$ are bounded away from both~0 and infinity.
Suppose also that $(\mu - \lambda) N^{1/2} \to \infty$ as $N \to \infty$, that $X_N(0)$ is non-random and that $X_N(0)(\mu-\lambda) \to \infty$ as
$N \to \infty$. Then, as $N \to \infty$,
\begin{equation} \label{eq.general}
(\mu-\lambda) T_N - \Big( \log N + 2 \log (\mu-\lambda) - \log\Big( 1 + \frac{(\mu-\lambda) N}{\lambda X_N(0)}\Big) - \log \mu -\log \lambda \Big) \to W,
\end{equation}
in distribution, where $W$ is a standard Gumbel variable.
Hence, as $N \to \infty$,
$$
\E  T_N = \frac{\log N + 2 \log (\mu-\lambda) - \log\big( 1 + \frac{(\mu-\lambda) N}{\lambda X_N(0)}\big) - \log \mu -\log \lambda + \gamma + o(1)}{\mu-\lambda}.
$$
\end{theorem}

\medskip

Observe that
$$
\log N + 2 \log (\mu-\lambda) - \log \left( 1 + \frac{(\mu-\lambda) N}{\lambda X_N(0)}\right)
= - \log \left( \frac{1}{N(\mu-\lambda)^2} + \frac{1}{\lambda X_N(0) (\mu-\lambda)} \right),
$$
which tends to infinity under the hypotheses of the theorem. The remaining terms, $\log \mu$ and $\log \lambda$, in the expression in~(\ref{eq.general}) are of constant order, and so Theorem~\ref{thm.main} implies that, for any fixed $\eps > 0$, the probability of extinction before time
$$
(1-\eps) \frac{\log N + 2 \log (\mu-\lambda) - \log\big( 1 + \frac{(\mu-\lambda) N}{\lambda X_N(0)}\big)}{\mu-\lambda}
$$
tends to zero, while the probability of extinction by time
$$
(1+\eps) \frac{\log N + 2 \log (\mu-\lambda) - \log\big( 1 + \frac{(\mu-\lambda) N}{\lambda X_N(0)}\big)}{\mu-\lambda}
$$
tends to~1.  This is an instance of the cut-off phenomenon (see for instance: Diaconis 1996; Levin, Peres and Wilmer 2009), where, for a Markov chain $(X_N(t))$, the total variation distance between the distribution of $X_N(t)$ and the stationary distribution moves rapidly from~1 to~0 over a time
interval much smaller than the
time to stationarity.  In our instance, the support of the stationary distribution is $\{0\}$, so the total variation distance at time $t$ is the
probability that the epidemic is not yet extinct.  This probability goes from near~1 to near~0 over a time interval of length of order $1/(\mu-\lambda)$,
whereas the expected extinction time, from a large enough initial state, is of order $\log[N(\mu-\lambda)^2]/(\mu-\lambda)$.  Thus the cut-off phenomenon
becomes less pronounced as $N(\mu-\lambda)^2$ tends more slowly to infinity, i.e., as we approach the critical regime.  See Figures~\ref{fig:unscaled}
and~\ref{fig:scaled}.

Another point to note is how the expected extinction time changes as $\mu-\lambda$ decreases, showing the transition between the
subcritical and critical regimes. When $\mu-\lambda$ is of constant order, the expected extinction time is $(1+o(1)) \log N / (\mu-\lambda)$;
as $(\mu - \lambda) N^{1/2}$ tends to infinity more and more slowly, the expected extinction time grows almost as large as $N^{1/2}$.

We next give versions of~(\ref{eq.general}) valid when $X_N(0)/N$ lies in certain ranges, assuming always that $(\mu(N)-\lambda(N)) N^{1/2} \to \infty$
and $X_N(0)(\mu(N) - \lambda(N)) \to \infty$.
One important special case is when $X_N(0)/N \to \alpha$, with $\alpha \in (0,1]$, when Theorem~\ref{thm.main} gives that, for a standard Gumbel random variable $W$,
\begin{equation} \label{eq.inter}
(\mu-\lambda) T_N - \big( \log N + 2\log (\mu-\lambda) + \log \alpha - \log(\lambda \alpha + \mu -\lambda) - \log \mu \big) \to W,
\end{equation}
in distribution, as $N \to \infty$. 

In general, (\ref{eq.general}) is the most that can be said if
$\mu -\lambda$ is of the same order as $X_N(0)/N$
(e.g., if both are constants).  On either side of this regime, the formula in (\ref{eq.general}) can be simplified.

For instance, if $X_N(0)/N(\mu-\lambda) \to \infty$ (so $X_N(0)/N$ is asymptotically of larger order than $\mu - \lambda$), 
then
\begin{equation} \label{eq.high}
(\mu-\lambda) T_N - \big( \log N + 2\log (\mu-\lambda) - \log \mu -\log \lambda \big) \to W,
\end{equation}
in distribution, as $N \to \infty$, where $W$ has the standard Gumbel distribution.
In (\ref{eq.high}), necessarily $\mu-\lambda \to 0$, so either of the terms $\log \mu$
and $\log \lambda$ could be replaced by the other.

Note that, for any $X$,
$$
\log N + \log (\mu-\lambda) - \log\big( 1 + \frac{(\mu-\lambda) N}{\lambda X}\big) -\log \lambda
= \log X - \log\big( \frac{\lambda X}{(\mu-\lambda)N} + 1 \big),
$$
and so an equivalent form of (\ref{eq.general}) is
\begin{equation} \label{eq.equiv}
(\mu-\lambda) T_N - \Big( \log X_N(0) + \log (\mu-\lambda) - \log\Big( 1 + \frac{\lambda X_N(0)}{(\mu-\lambda) N}\Big) - \log \mu \Big) \to W.
\end{equation}
It follows that, if $X_N(0)/N(\mu - \lambda) \to 0$ (so $X_N(0)/N$ is asymptotically of smaller order than $\mu - \lambda$)
and $X_N(0)(\mu-\lambda) \to \infty$, then
\begin{equation} \label{eq.low}
(\mu-\lambda) T_N - \big( \log X_N(0) + \log (\mu-\lambda) - \log \mu \big) \to W,
\end{equation}
in distribution, as $N \to \infty$. 

We observe that the asymptotic formula for the distribution of $T_N$ in~(\ref{eq.low})
is independent of~$N$, while that in~(\ref{eq.high}) is independent of~$X_N(0)$.  An explanation for the first of these phenomena 
is that, in this regime, the quadratic terms in the drift are smaller than the linear ones, and the logistic process behaves essentially identically to a linear birth-and-death chain
with birth rate $\lambda$ and death rate $\mu$. In Section~\ref{sec:final}, we show that the logistic correction to the birth rate
(i.e., the term $-\lambda X(t)^2/N$), does not affect the asymptotics of the remaining time to extinction.

If $\mu-\lambda \to 0$, then, for large enough $X_N(0)$ (such that $X_N(0) \gg (\mu-\lambda)N$), we are in the regime covered by~(\ref{eq.high}),
where the asymptotic distribution of the time to extinction does not depend on the starting state.
To explain this, we give an informal description of the typical course of the epidemic in the case where $\mu-\lambda \to 0$, and we start in some
large state, say with $X_N(0) = \lceil \alpha N \rceil$ and $0 < \alpha \le 1$.

For such a regime, in the initial phase of the epidemic, the number $X_N(t)$ of infectives very rapidly
drops -- in time $o(1/(\mu-\lambda))$ -- until it reaches states of the same order as $(\mu-\lambda)N$.
The majority of the duration of the epidemic -- asymptotically $\log (N(\mu-\lambda)^2))/(\mu-\lambda)$ --
is spent in an intermediate phase, getting from there to states of the same order as $(\mu-\lambda)^{-1}$; the time taken to cross this gap is very
well concentrated around the value derived from the approximating differential equation.
Most of the variability of the time to extinction comes from the final phase of the epidemic, starting when $X_N(t)$ is about the order of
$(\mu-\lambda)^{-1}$; for this final phase, the differential equation is no longer an adequate guide to the behaviour of the stochastic process, and
instead $X_N(t)$ is well approximated by a linear birth-and-death chain.  The expected time for the final phase, starting from a state of order
$(\mu-\lambda)^{-1}$, is on the order of $1/(\mu-\lambda)$, and the standard deviation is of the same order.  (The choice of where to draw the
line between the intermediate and the final phase is somewhat arbitrary: our results essentially show that the approximation by a linear
birth-and-death chain is good starting from any state below the order of $(\mu-\lambda)N$.)

Note that the description above relies on having $(\mu-\lambda)N \gg (\mu-\lambda)^{-1}$. If $\mu-\lambda =o(N^{-1/2})$, then the situation is
completely different: the time to extinction is essentially distributed as in the case $\mu=\lambda$.

The assumption in Theorem~\ref{thm.main} that $X_N(0)(\mu-\lambda) \to \infty$ is necessary for the conclusion to hold; otherwise the variability in the
extinction time is not as large as is given by the Gumbel distribution.
We give more details for the case where this assumption is not satisfied at the end of Section~\ref{sec:final}.

Sections~\ref{sec:final}-\ref{sec.proof} are devoted to the proof of Theorem~\ref{thm.main}.
We track the epidemic process through three phases, roughly corresponding to the three phases mentioned in the informal
description above.  For some regimes, not all the phases are necessary, and we tackle them in reverse order, starting with the final phase of the epidemic. Our intermediate results are stated in terms of a function $\omega(N)$, which tends to infinity suitably slowly; for convenience, we
specify throughout that
\begin{equation} \label{eq:omega}
\omega(N) = (\mu (N)-\lambda (N))^{1/4}N^{1/8}.
\end{equation}

We treat the final phase in Section~\ref{sec:final}.  Here, we start from a state below $N^{1/2} \omega(N)$, and show that, from this point on,
$X_N(t)$ is well approximated by a linear birth-and-death chain with the same parameters. Since the distribution of the extinction time
for a linear birth-and-death chain is known explicitly, this enables us to analyse very precisely the behaviour of the logistic chain. This phase covers the stage of the epidemic giving rise to the randomness in the time to extinction.  An alternative
way to view the final phase of the epidemic is to approximate it by a branching process, where each initially infected individual sparks a brief small epidemic
within the population, and these various small epidemics do not interact significantly.  The time to extinction is then the maximum of the durations of these
small epidemics, and this explains the appearance in our formulae of the standard Gumbel distribution, which typically arises as the maximum of a number of
independent samples from a given distribution.
Note again that the exact break point between the final and intermediate phases is somewhat arbitrary.  We do need to start the final phase in a state well
below $N(\mu - \lambda)$, so that the logistic effects can be ignored, and it is helpful to us to start slightly smaller yet;
on the other hand we do need the initial state larger than $(\mu - \lambda)^{-1}$ for the formula involving the Gumbel distribution to apply.

The intermediate phase is covered in Section~\ref{sec:intermediate}.  Here,
effectively, we prove Theorem~\ref{thm.main} under the additional assumption that $X_N(0) \le (\mu-\lambda)N \omega(N)$.
We show that the scaled process $(X_N(t)/N)$ stays close to the solution of the differential equation~(\ref{eq.diff-eq}) for a (deterministic) period of time
until there are about $N^{1/2} \omega(N)$ infective individuals (from which point the analysis for the final phase can be invoked).

In Section~\ref{sec:initial}, we provide a fairly crude upper bound on the duration of the initial phase of the epidemic, starting from any state and
reaching a state of order about $N(\mu-\lambda)\omega(N)$.  The length of this phase is negligible compared to the overall duration of the epidemic, or even the fluctuations in the overall duration, so greater precision is not necessary.

In Section~\ref{sec.proof}, we combine our results to prove Theorem~\ref{thm.main}.

Our results have some bearing on the critical regime, where $|\mu-\lambda| = O(N^{-1/2})$.  In particular, the methods of Section~\ref{sec:initial}
can be used to show that the expected time for the epidemic starting from an arbitrary state to reach a state of size about $N^{1/2}$ is
of order at most $N^{1/2}$.  Dolgoarshinnykh and Lalley~(2006) show that, in this regime, the scaled logistic process starting from a
state of order $N^{1/2}$ converges in law to an ``attenuated'' Feller diffusion.  One consequence is that the time to
extinction from states of size  about $N^{1/2}$ is of order $N^{1/2}$ (and is not well-concentrated).
We discuss the critical regime briefly in the short Section~\ref{sec.critical}, but make no attempt to provide precise results.

In Section~\ref{sec.total}, we consider the total number $C_N$ of new cases (i.e., infection events) over the duration of the epidemic.
Theorem~\ref{thm.total} provides a precise estimate, valid throughout the subcritical regime, of the expectation of $C_N$ of new cases,
and states that $C_N$ is well-concentrated around its mean, via an estimate of the variance of $C_N$.
One consequence of this result is that, if $X_N(0) \gg (\mu-\lambda)N$, then most of the new cases occur during the short first
phase of the epidemic, i.e., before $X_N(t)$ has dropped to around $(\mu-\lambda)N$.

The total number of new cases in the SIS logistic epidemic is studied in detail by Kessler~(2008), for the full range of parameter values.
In the subcritical regime, Kessler~(2008) gives an asymptotic formula agreeing with ours for the expectation of $C_N$ when $X_N(0)$ is
of order~$N$, and discusses other cases, including ones where $\mu-\lambda = \delta N^{-1/2}$ and $\delta$ is large.  He also estimates the
asymptotic distribution of $C_N$ in subcritical, critical and supercritical regimes, but only in the case
$X_N(0) = 1$.  Our results show that, provided $X_N(0)(\mu-\lambda)$ tends to infinity, $C_N$ is well-concentrated around its expectation.

In Section~\ref{sec.finite}, we present numerical methods to treat fixed values of~$N$.
All our theoretical results concern limiting behaviour as the population size $N$ tends to infinity, and in many places it
is important that terms such as
$\log N$ (or indeed functions potentially growing more slowly) are much larger than constants.  It is not apparent that our results have
any bearing on ``human-size'' populations: we address this issue by performing numerical calculations for a range of values of $N$ and appropriate
values of $\lambda$ and $\mu$.  As we explain in Section~\ref{sec.finite}, it is more efficient to estimate the distribution of (for instance) the
extinction time by numerical integration, as opposed to using Monte Carlo methods.  We see good agreement between asymptotic results and simulation
for temporal behaviour even for $N =10$ and $N=100$, when typically one does not see this until $N=1000$ for epidemic models (see, e.g., Demiris
and O'Neill 2006).

We conclude, in Section~\ref{sec.mean-extinction}, by expanding on our observations about the behaviour of a barely subcritical epidemic.  We discuss
in particular those features that we expect to carry over to more complex models, or to real-world epidemics, especially
the resemblance to a random walk for a period before extinction, and the weakening of the cut-off phenomenon as we approach criticality.
In Appendix~A, we present some data from real epidemics, and make some very tentative connections between our hypotheses and the observations.




\section{Final phase: approximation by linear birth-and-death chains} \label{sec:final}

Suppose that $\mu=\mu(N)$ and $\lambda=\lambda(N)$ are bounded away from both~0 and infinity, and that $X_N(0)$ is non-random with
$X_N(0) (\mu - \lambda) \to \infty$ and $X_N(0) \le \omega(N) N^{1/2}$, where $\omega (N) = (\mu - \lambda)^{1/4} N^{1/8}$, as in (\ref{eq:omega}).
We will show that, with such an initial state,
$X_N(t)$ is well approximated until extinction by a pair of linear birth-and-death chains. The assumption that $X_N(0)(\mu-\lambda) \to \infty$
ensures that the randomness in the extinction time of the approximating linear birth-and-death chains has a Gumbel distribution.

We will prove the following lemma.

\begin{lemma} \label{lem.late}
Suppose that $(\mu-\lambda)N^{1/2} \to \infty$.  Set $\omega (N) = \big( (\mu-\lambda)N^{1/2} \big)^{1/4}$, and
assume that $X_N(0)(\mu-\lambda) \to \infty$ and $X_N(0) \le 2N^{1/2} \omega(N)$.
Then
$$
(\mu - \lambda) T_N - \big( \log X_N(0) + \log (\mu-\lambda) - \log \mu \big) \to W,
$$
in distribution, as $N \to \infty$, where $W$ has the standard Gumbel distribution.
\end{lemma}

Note that this result is the same as the special case of Theorem~\ref{thm.main} covered by (\ref{eq.low}),
under the more restrictive hypothesis that $X_N(0) \le 2N^{1/2} \omega(N) = 2N^{5/8}(\mu-\lambda)^{1/4}$ instead of $X_N(0) = o((\mu-\lambda)N)$.
Later results will supply the conclusion of Theorem~\ref{thm.main} with no upper bounds on the starting state.

\medskip

For a birth-and-death chain $(B(t))$ on $\Z_+$ with a unique absorbing state at $0$, let $T^B$ be the extinction time, that is $T^B = \inf \{t: B(t) = 0\}$ is the time when $(B(t))$ gets absorbed at $0$. Thus the extinction time of $(X_N(t))$ is $T_N = T^{X_N}$.

\smallskip

Let $(Y(t))_{t \ge 0}$ be a linear birth-and-death chain with birth rate $\lambda$ and death rate $\mu$, so its transition rates from state $Y \in \Z_+$ are given by
\begin{eqnarray*}
Y &\to& Y+1 \quad \mbox{ at rate } \quad \lambda Y, \\
Y &\to& Y-1 \quad \mbox{ at rate } \quad \mu Y.
\end{eqnarray*}

Assume that $Y(0)$ is non-random.
It is known -- see for instance (2.4.23) in the book of Renshaw~(2011) -- that, for $t \ge 0$ and $\mu \not= \lambda$,
\begin{equation} \label{eq.lbdc}
\Pr (T^Y \le t) = \Pr (Y(t) = 0)
= \left( \frac{\mu  - \mu e^{-(\mu-\lambda)t}}{\mu - \lambda e^{-(\mu-\lambda)t} } \right)^{Y(0)}
\!= \left( 1 - \frac{(\mu-\lambda) e^{-(\mu - \lambda) t}}{\mu - \lambda e^{-(\mu-\lambda)t}} \right)^{Y(0)}\!.
\end{equation}

We will write $(Y_N(t))$ to denote a linear birth-and-death chain with birth rate $\lambda(N)$ and death rate $\mu(N)$.
Suppose that $Y_N(0) = X_N(0)$, where $X_N(0)(\mu-\lambda) \to \infty$ as $N \to \infty$.
For each fixed $w \in \R$, we set
$$
t_w = t_w(\mu,\lambda,X_N(0)) = \frac{\log X_N(0) + \log (\mu-\lambda) - \log \mu + w}{\mu-\lambda},
$$
and note that $t_w > 0$ for sufficiently large~$N$.  Restricting to those $N$ for which $t_w$ is indeed positive, we have
that $e^{-(\mu-\lambda)t_w} = \mu e^{-w} / (\mu-\lambda)X_N(0)$. Hence, from (\ref{eq.lbdc}), the extinction time $T^{Y_N}$ satisfies
\begin{eqnarray*}
\Pr (T^{Y_N} \le t_w) &=& \left( 1 - \frac{\mu e^{-w}/X_N(0)}{\mu - \lambda \mu e^{-w}/(\mu-\lambda)X_N(0)} \right)^{X_N(0)} \\
&=& \left( 1 - \frac{e^{-w}}{X_N(0) - \lambda e^{-w}/(\mu-\lambda)} \right)^{X_N(0)} \to e^{-e^{-w}},
\end{eqnarray*}
as $N \to \infty$, since $X_N(0)(\mu-\lambda) \to \infty$.  This can be written as
\begin{equation} \label{eq.lbdc2}
(\mu - \lambda) T^{Y_N} - \big( \log X_N(0) + \log (\mu-\lambda) - \log \mu \big) \to W,
\end{equation}
in distribution, as $N \to \infty$, where $W$ has the standard Gumbel distribution.

\medskip

The plan of the proof of Lemma~\ref{lem.late} is to sandwich the logistic process $(X_N(t))$ between two linear birth-and-death chains, the upper of which is $(Y_N(t))$. The upper bound on
$X_N(0)$ ensures that $X_N(0)/N(\mu - \lambda) \to 0$, so ``logistic effects'' in the drift become negligible and so the linear birth-and-death chains approximate $(X_N(t))$ well.
Our argument is a little crude, in that
the birth rate of the lower of the two birth-and-death chains is significantly below that of the logistic process for most of the phase, and this is
why we need the stronger hypothesis $X_N(0) \le 2N^{1/2} \omega(N)$, rather than just $X_N(0)/N(\mu - \lambda) \to 0$, and the precise form of $\omega$ matters here.

We will use the following result about birth-and-death chains with a higher rate of deaths than births: we omit the routine proof.

\begin{lemma} \label{lem.escape}
Let $(B(t))$ be a birth-and-death chain on $\Z_+$, with $B(0)$ non-random.
Suppose that, from any state, the probability that the next transition is upwards is at most $p$, and the probability that the
next transition is downwards is at least $q >p$.

For any state $B > B(0)$, the probability that the chain $(B(t))$ reaches $B$ before it reaches~0 is at most
$$
\frac{(q/p)^{B(0)} - 1}{(q/p)^B - 1} \le \left( \frac{p}{q} \right)^{B-B(0)} \le \exp \left( - (1-p/q) (B-B(0))\right).
$$
\end{lemma}



\medskip

\begin{proof}[Proof of Lemma~\ref{lem.late}]
We couple three Markov chains: one is the logistic process $(X_N(t))$, another is the linear birth-and-death chain
$(Y_N(t))$ with the same parameters $(\lambda (N),\mu (N))$ as $(X_N(t))$, and the third is a linear birth-and-death chain $(Z_N(t))$  with
parameters $(\lambda' (N), \mu(N))$ where $\lambda'(N) = \lambda (N) (1-2X_N(0)/N)$.  We let $Y_N(0) = X_N(0) = Z_N(0)$.
Let
$\tau_N$ be the first time that either $X_N(t)=0$ or $X_N(t) = 2X_N(0)$.  The birth rate of $(X_N(t))$ when in state $X$ is $\lambda X (1- X/N)$, which for
$t \le \tau_N$ is sandwiched between the birth rates of the two linear birth-and-death chains in the same state.
For each $N$, we may thus construct a coupling such that
$Z_N(t) \le X_N(t) \le Y_N(t)$ for all $t\le \tau_N$.
The rule is that,
if any two chains are in the same state, then they make jumps together as far as possible; otherwise two chains in different states make jumps
independently according to their given transition rates, and so they a.s.\ do not jump
simultaneously (so they do not cross).
With this coupling, on the event that $X_N(\tau_N) = 0$ (i.e., $(X_N(t))$ reaches~0 before it reaches
the upper boundary $2X_N(0)$), $T^{Z_N} \le T_N \le T^{Y_N}$.

For each fixed $w$ and $N$, we choose $v=v(w,N)$ so that $t_w(\mu,\lambda,X_N(0)) = t_v(\mu,\lambda',X_N(0))$, i.e.,
$$
\frac{\log X_N(0) + \log (\mu-\lambda) - \log \mu + w}{\mu-\lambda} = \frac{\log X_N(0) + \log (\mu-\lambda') - \log \mu + v}{\mu-\lambda'}.
$$
This translates to
\begin{equation} \label{eq.hot}
v(w,N) - w = \frac{\lambda-\lambda'}{\mu-\lambda}\left( \log X_N(0) + \log(\mu-\lambda) -\log \mu + w\right) -
\log \left(1 + \frac{\lambda-\lambda'}{\mu-\lambda} \right).
\end{equation}
We observe that
\begin{equation} \label{eq.hut}
\frac{\lambda - \lambda'}{\mu-\lambda} = \frac{2X_N(0) \lambda}{N(\mu-\lambda)} \le \frac{4\lambda \omega(N)}{N^{1/2}(\mu-\lambda)}
= \frac{4\lambda}{\omega(N)^3}.
\end{equation}
Also we have, for $N$ sufficiently large,
$$
\log (X_N(0)(\mu-\lambda)) \le (X_N(0)(\mu-\lambda))^{1/2} \le \big( 2N^{1/2}\omega(N) N^{-1/2}\omega(N)^4 \big)^{1/2} = 2\omega(N)^{5/2}.
$$
Therefore, for each fixed $w$, we have both
$$
\frac{\lambda-\lambda'}{\mu-\lambda}\left( \log X_N(0) + \log(\mu-\lambda) -\log \mu + w\right)
\le \frac{2\lambda}{\omega(N)^3} \left( 2\omega(N)^{5/2} + O(1) \right) = o(1)
$$
and, by (\ref{eq.hut}),
$$
\log \left(1 + \frac{\lambda-\lambda'}{\mu-\lambda} \right) = o(1),
$$
and so, from (\ref{eq.hot}), $|v(w,N)-w| = o(1)$.
In other words, $v(w,N) \to w$ as $N\to \infty$, for each fixed $w$.

Thus
\begin{eqnarray*}
\lefteqn{|\Pr(T^{Z_N} \le t_w(\mu,\lambda,X_N(0))) - e^{-e^{-w}}|} \\
&=& |\Pr(T^{Z_N} \le t_v(\mu,\lambda',X_N(0))) - e^{-e^{-w}}| \\
&\le& | \Pr (T^{Z_N} \le t_v(\mu,\lambda',X_N(0))) - e^{-e^{-v(w,N)}}| + | e^{-e^{-v(w,N)}} - e^{-e^{-w}}| \to 0
\end{eqnarray*}
as $N \to \infty$.  Here we used (\ref{eq.lbdc2}) applied to $(Z_N(t))$ (note that $X_N(0)(\mu-\lambda') > X_N(0)(\mu-\lambda)$, which tends to infinity).

Also by~(\ref{eq.lbdc2}), as $N \to \infty$, 
$$
\Pr(T^{Y_N} \le t_w(\mu,\lambda,X_N(0))) \to e^{-e^{-w}}.
$$

As $(X_N(t))$ is sandwiched between $(Y_N(t))$ and $(Z_N(t))$ for all times $t$, on the event $A:=\{X_N(\tau_N) = 0\}$,
we see that
\begin{eqnarray*}
\Pr (\{T^{Y_N} \le t \} \cap A) \le \Pr (\{T_N \le t\} \cap A) \le \Pr ( \{T^{Z_N} \le t \} \cap A),
\end{eqnarray*}
and in particular this holds with $t =  t_w(\mu,\lambda,X_N(0))$.
By Lemma~\ref{lem.escape} with $p = \lambda/(\lambda + \mu) = 1-q$ and $B(0) = X_N(0)$, 
$\Pr (\overline{A}) \le e^{-(\mu-\lambda)X_N(0)/\mu} = o(1)$,
since $(\mu-\lambda)X_N(0) \to \infty$. Hence, as $N \to \infty$,
$$
\Pr(T_N \le t_w(\mu,\lambda,X_N(0))) \to e^{-e^{-w}}.
$$
Equivalently, as $N \to \infty$,
$$
(\mu - \lambda) T_N - \big( \log X_N(0) + \log (\mu-\lambda) - \log \mu \big) \to W,
$$
in distribution, where $W$ has the standard Gumbel distribution, as claimed.
\end{proof}

In the case where $(\mu-\lambda)X_N(0)$ does not tend to infinity, we can use a similar argument to show that the distribution of
the extinction time of $(X_N(t))$ is asymptotically the same as that of the linear birth-and-death chain with the same
parameters.  We give a brief sketch of the argument in the case where $(\mu-\lambda)X_N(0) \to 0$.

For a linear birth-and-death chain $(Y_N(t))$ with parameters $(\lambda,\mu)$ and $Y_N(0) = X_N(0) = o(1/(\mu-\lambda))$,
Lemma~\ref{lem.escape} shows that the probability of the event $A$ that $(Y_N(t))$ never reaches $N^{1/2}$ before extinction is $1 - o(1)$.
Accordingly, we consider also a linear birth-and-death chain $(Z_N(t))$ with birth rate equal to $\lambda' = \lambda(1-N^{-1/2})$.  As in the proof of
Lemma~\ref{lem.late}, we may couple our three processes so that $Z_N(t) \le X_N(t) \le Y_N(t)$ for all $t$, on the event $A$.

It can be seen from (\ref{eq.lbdc}) that, if $(\mu-\lambda)X_N(0) \to 0$ and $X_N(0) \to \infty$, then for any $v \in (0,\infty)$,
\begin{equation} \label{eq:hat}
\Pr \big(T^{Y_N} \le v X_N(0) /\mu\big) \to e^{-1/v} \quad \mbox{ as } N \to \infty.
\end{equation}
Note that (\ref{eq:hat}) does not depend on $\mu-\lambda$, provided that $(\mu-\lambda)X_N(0) \to 0$.
Note also that $(\lambda-\lambda')X_N(0) = \lambda N^{-1/2}X_N(0) = o(N^{-1/2}/(\mu-\lambda)) = o(1)$, and so
$(\mu-\lambda')X_N(0) \to 0$ whenever $(\mu-\lambda)X_N(0) \to 0$.  Therefore (\ref{eq:hat}) holds with $T^{Y_N}$ replaced by $T^{Z_N}$,
and hence also $\Pr (T_N \le v X_N(0)/\mu) \to e^{-1/v}$.

We can also consider the case where the epidemic starts with a single infective: if $\mu-\lambda \to 0$ and $X_N(0) = 1$, then, for any
$u \in (0, \infty)$,
$$
\Pr \big(T^{X_N} \le u /\mu\big) \to \frac{u}{1+ u} \quad \mbox{ as } N \to \infty.
$$

\section{Intermediate phase: differential equation approximation} \label{sec:intermediate}

For any $\alpha \in [0,1]$, the differential equation~(\ref{eq.diff-eq}) subject to initial condition $x(0) = \alpha$ has an explicit solution
\begin{equation}
\label{eq.solution}
x(t) = \frac{ \alpha (\mu-\lambda) e^{-(\mu-\lambda)t}}{\mu-\lambda + \alpha \lambda(1- e^{-(\mu-\lambda)t})}, \quad \quad t \ge 0.
\end{equation}
For fixed $0 < \alpha \le 1$, the inverse of the function $x(t)$ is given by
\begin{equation}
\label{eq.ST}
t_{\alpha}(x) = \frac{s(x) - s(\alpha)}{\mu-\lambda}; \quad \mbox{ where } s(x) = \log \left( 1 + \frac{\lambda}{\mu-\lambda} x \right) - \log x,
\end{equation}
for $0 < x \le \alpha$.

We also note for future reference that
$x(t) \le x(0) e^{-(\mu-\lambda)t}$, and therefore, for any $t \ge 0$,
\begin{equation} \label{eq.integralx}
\int_0^t x(s) \, dt \le \frac{x(0)}{\mu-\lambda}.
\end{equation}

As in (\ref{eq:omega}), we set $\omega(N) = (\mu-\lambda)^{1/4} N^{1/8}$, and suppose that $N^{1/2} \omega (N) \le X_N(0) \le (\mu-\lambda) N \omega(N)$. Let $X^* = X^*(N) = N^{1/2} \omega(N)$.
We will show that $X_N(t)/N$ is well approximated by the solution $x(t)$ to the differential equation~(\ref{eq.diff-eq}) with $x(0) = X_N(0)/N$,
at least until the time $t^* = t_{X_N(0)/N}(X^*/N)$ when $x(t^*) = X^*/N$.
It will then follow that $X_N(t^*)$ is close to $X^*$ with probability $1-o(1)$ as $N \to \infty$.
The total time to extinction will then be obtained by adding $t^*$ to the time to extinction from a state very near to $X^*$,
which is covered in Lemma~\ref{lem.late}.


\medskip



To be precise, we will prove the following result.

\begin{lemma} \label{lem.slow}
Suppose $(\mu-\lambda) N^{1/2} \to \infty$ as $N\to \infty$.
Set $\omega(N) = (\mu-\lambda)^{1/4} N^{1/8}$, and $X^*=X^*(N) = N^{1/2} \omega(N)$.
Suppose $X^* \le X_N(0)  \le \omega(N) (\mu-\lambda) N$.
Then
$$
\Pr \left( \left| X_N(t^*) - X^* \right| > \omega(N)^{-1/3} X^* \right) = o(1),
$$
where
$$
t^* = t_{X_N(0)/N}(X^*/N) = \frac{s(X^*/N) - s(X_N(0)/N)}{\mu-\lambda}.
$$

Moreover, we have
\begin{equation} \label{eq.number}
t^* = \frac{1}{\mu-\lambda} \left( \log X_N(0) - \log(X^*) - \log \left(1+ \frac{\lambda}{\mu-\lambda} \frac{X_N(0)}{N} \right) + o(1) \right).
\end{equation}
\end{lemma}

The final assertion in the statement follows immediately from the expression in~(\ref{eq.ST}) for $s(x)$, since, with the assumptions given,
$\log \big(1 + \frac{\lambda}{\mu-\lambda} \frac{X^*}{N} \big) = o(1)$.

\medskip

To prove Lemma~\ref{lem.slow}, we will use standard martingale techniques, taking advantage of the special error-correcting nature of the drift in process $X_N(t)$, thanks to which errors in the approximation do not accumulate much over time.
Let $x(t)$ be as in~(\ref{eq.solution}), with $\alpha = x(0) = X_N(0)/N$.  Now set
$$
T = \inf \left\{ t \ge 0: |N^{-1} X_N(t) - x(t)| > 2 \sqrt{X_N(0) (\omega (N))^{1/4} (\lambda + \mu)/N^2(\mu - \lambda)} \right\}.
$$
Note that it will suffice to show that $\Pr (T \le t^*) = o(1)$. This is because, if $T > t^*$, then
\begin{eqnarray}
|X_N(t^*) - X^*| & \le & N \sup_{t \le t^*} \Big |\frac{X_N(t)}{N} - x(t) \Big |
\le 2 \sqrt \frac{X_N(0) \omega^{1/4} (\lambda + \mu)}{\mu-\lambda} \nonumber\\
& \le & \sqrt{\omega^{5/4} N (\mu + \lambda )} = (\mu + \lambda )^{1/2}N^{1/2} \omega^{5/8} \nonumber \\
& = & o(X^* \omega^{-1/3}), \label{eq-o}
\end{eqnarray}
since we have assumed that $X_N(0) \le \omega (N) (\mu - \lambda)N$ and since $X^* = N^{1/2} \omega (N)$.

We write, as is standard,
$$
x(t) = x(0) + \int_0^t f(x(s)) \, ds,
$$
where $f(x) = \lambda x (1-x) - \mu x = - (\mu-\lambda) x - \lambda x^2$.

Also by standard theory,
\begin{equation}
\label{eq.dynkin}
\frac{X_N(t)}{N} = \frac{X_N(0)}{N} - (\mu - \lambda) \int_0^t \frac{X_N(s)}{N} \, ds - \lambda \int_0^t \Big (\frac{X_N(s)}{N} \Big )^2 \, ds + M_N(t),
\end{equation}
where $(M_N(t))_{t\ge 0}$ is a zero-mean martingale.

Setting $e_N(t) = X_N(t)/N - x(t)$, it follows that
\begin{eqnarray} \label{eq.transform2}
e_N(t) 
&=& - \int_0^t  e_N(s) \Big [(\mu - \lambda) + \lambda \Big (\frac{X_N(s)}{N}+ x(s) \Big ) \Big ] \, ds + M_N(t).
\end{eqnarray}

To bound $e_N(t)$, we use the following simple lemma.  For future applications (e.g., in forthcoming work by Lopes and Luczak on the SIS logistic
competition model), we state it in a slightly more general form than needed here.

\begin{lemma} \label{lem.narrowing}
Fix a time $\tau_0$, and let $m: [0,\tau_0] \to \R$ and $r, v: [0,\tau_0] \to \R^+$ be c\`adl\`ag functions, where $v$ is decreasing, and suppose that
$u: [0,\tau_0] \to \R$ is a c\`adl\`ag function satisfying
$$
u(t) = m(t) - v(t) \int_0^t r(s) u(s) \, ds,
$$
for $0\le t \le \tau_0$.  Then
$$
\sup_{t \le \tau_0} |u(t)| \le 2 \sup_{t\le \tau_0} |m(t)|.
$$
\end{lemma}

\begin{proof}
Let $M = \sup_{t \le \tau_0} |m(t)|$.  Choose any $\tau \in [0,\tau_0]$, and suppose without loss of generality that $u(\tau) \ge 0$.
If $u(t) \ge 0$ for all $t \le \tau$, then we certainly have $u(\tau) \le m(\tau) \le M$.  Otherwise,
let $\sigma = \sup\{ t \le \tau : u(t) < 0\} > 0$, and observe that $\lim_{s \to \sigma-} u(s) \le 0$
and $u(s) \ge 0$ for $\sigma < s \le \tau$, and so $\int_{\sigma}^\tau r(s)u(s) \, ds \ge 0$.
We may therefore write
\begin{eqnarray*}
u(\tau) &=& m(\tau) - v(\tau) \int_0^\sigma r(s) u(s) \, ds - v(\tau) \int_{\sigma}^\tau r(s) u(s) \, ds \\
&=& m(\tau) - \lim_{t \to \sigma-} (m(t) - u(t)) \frac{v(\tau)}{v(t)} - v(\tau) \int_{\sigma}^\tau r(s) u(s) \, ds  \\
&\le & M + M \lim_{t\to\sigma-} \frac{v(\tau)}{v(t)} + \lim_{t\to \sigma-} u(t) \frac{v(\tau)} {v(t)} - v(\tau) \int_{\sigma}^\tau r(s) u(s) \, ds \\
&\le & M + M + 0 + 0 = 2M.
\end{eqnarray*}
Hence $|u(t)| \le 2M$ for all $t \le \tau_0$, as required.
\end{proof}

We apply Lemma~\ref{lem.narrowing} with $\tau_0 = t^*$, $u(t) = e_N(t)$, $m(t) = M_N(t)$, $v(t) = 1$, and
$r(s) = (\mu - \lambda) + \lambda (X_N(s)/N + x(s))$.
The hypotheses of the lemma are satisfied since $\mu > \lambda$, and so we have
\begin{equation} \label{eq.bound}
\sup_{t \le t^*} |e_N(t)| \le 2 \, \sup_{t \le t^*} |M_N(t)|. 
\end{equation}
Therefore,
\begin{equation}
\label{eq-T}
\Pr (T \le t^*) \le \Pr \Big (\sup_{t \le t^*} |M_N(t)| > \sqrt{X_N(0) (\omega (N))^{1/4} (\lambda + \mu)/N^2(\mu - \lambda)} \Big ).
\end{equation}

To bound $|M_N(t)|$, we use a standard exponential martingale argument.
%
Let $q_N^1 (x) = \lambda N x(1-x)$ and $q_N^{-1}(x) = \mu N x$ denote the rates of transition of $N^{-1}X_N(s)$, by $1/N$ and $-1/N$, respectively.
For $\theta \in \R$, we define $V_N^{\theta} (t)$ by
\begin{eqnarray}
V_N^{\theta} (t) & = & \exp \Big (\theta N^{-1} (X_N(t) - X_N(0)) - \int_0^t \sum_j q_N^j (N^{-1} X_N(s)) (e^{\theta N^{-1}j}-1) \, ds   \Big ) \nonumber \\
& = & \exp \Big (\theta M_N(t) - \int_0^t \sum_j q_N^j (N^{-1} X_N(s)) (e^{\theta N^{-1}j}-1 - \theta N^{-1} j) \, ds   \Big ) \label{eq.expo-mart}.
\end{eqnarray}
The process $(V_N^{\theta}(t))$ is a mean~1 martingale.
Using that $e^z - 1 - z = z^2 \int_0^1 e^{rz} (1-r) \, dr \le \frac12 z^2 e^{|z|}$, we
see that
$$
V_N^{\theta} (t) \ge \exp \Big ( \theta M_N(t) - \frac{\theta^2}{2N^2} e^{|\theta |/N}  \int_0^t \sum_j q_N^j (N^{-1} X_N(s))   \, ds  \Big ).
$$
Assume that $|\theta | \le N \log 2$. Let $T_1 = \inf  \{t \ge 0: X_N(t) > 2Nx(t) \}$;
then for $t \le T_1$,
\begin{eqnarray}
V_N^{\theta}(t)
& \ge & \exp \Big  (  \theta M_N(t) - \frac{2\theta^2}{N} (\lambda + \mu) \int_0^t x(s) \, ds    \Big ) \nonumber \\
& \ge & \exp \Big  (  \theta M_N(t) - \frac{2 \theta^2(\lambda + \mu)X_N(0)}{N^2(\mu - \lambda) }\Big ), \label{eq.em-bound}
\end{eqnarray}
by (\ref{eq.integralx}).
%

For $\delta \in \R$, let $T^+(\delta) = \inf \{t \ge 0: M_N(t) > \delta \}$, and let
$T^- (\delta) = \inf \{t \ge 0: M_N(t) < -\delta \}$.
On the event $\{T^+(\delta) \le T_1 \}$,
$$
V_N^\theta (T^+(\delta)) \ge \exp \Big (\theta \delta - \frac{2 \theta^2(\lambda + \mu)X_N(0)}{N^2(\mu - \lambda) } \Big ).
$$
By optional stopping and the Markov inequality,
$$
\Pr (T^+(\delta) \le  T_1  ) \le  \exp \Big (-\theta \delta + \frac{2 \theta^2 X_N(0)(\lambda + \mu)}{N^2(\mu - \lambda )} \Big ).
$$

Choosing $\theta =  \frac14 \delta  N^2(\mu - \lambda )/ X_N(0) (\lambda + \mu )$, we have $|\theta | \le  N \log 2$ for sufficiently large~$N$,
provided $\delta = o(X_N(0)/N(\mu-\lambda))$.
We then obtain
$$
\Pr (T^+(\delta) \le  T_1  ) \le e^{-\delta^2N^2 (\mu - \lambda) /8X_N(0) (\lambda + \mu)},
$$
and, similarly,
$$
\Pr (T^-(\delta) \le  T_1 ) \le e^{-\delta^2N^2 (\mu - \lambda) /8X_N(0) (\lambda + \mu)}.
$$
It follows that 
$$
\Pr ( \sup_{t \le t^* \wedge T_1} |M_N(t)| > \delta ) \le 2 e^{-\delta^2N^2 (\mu - \lambda) /8X_N(0) (\lambda + \mu)}.
$$
Take $\delta = \sqrt{X_N(0) \psi (\lambda + \mu)/N^2(\mu - \lambda)}$, for some $\psi \le \sqrt N$; this choice guarantees that $\delta = o(X_N(0)/N(\mu-\lambda))$,
since it is equivalent to $\psi = o(X_N(0)/(\mu - \lambda))$, and we have assumed that $X_N(0)/ N^{1/2} \to \infty$. We thus obtain
$$
\Pr \Big ( \sup_{t \le t^* \wedge T_1} |M_N (t)| > \sqrt{\frac{X_N(0) \psi (\lambda + \mu)}{N^2(\mu - \lambda)}} \Big ) \le 2 e^{-\psi/8}.
$$
If we choose $\psi = (\omega (n))^{1/4}$, then $\psi \le N^{1/2}$ for $N$ large enough, since, by~(\ref{eq:omega}),
$\omega(N) = (\mu (N)-\lambda (N))^{1/4}N^{1/8} =O(N^{1/8})$.
Let $\delta_0 = \sqrt{\frac{X_N(0) \omega^{1/4} (\lambda + \mu)}{N^2(\mu - \lambda)}}$, and let $T_0 = T^+(\delta_0) \wedge T^-(\delta_0)$.
Then, using~\eqref{eq-T},
\begin{eqnarray*}
\Pr (T \le t^*) \le \Pr (T_0 \le t^*) \le \Pr (T_0 \le t^* \wedge T_1) + \Pr (T_1 \le t^* \wedge T_0) \le 2 e^{-\omega (N)^{1/4}/8},
\end{eqnarray*}
since we showed in~\eqref{eq-o} that $\delta_0 = o(X^*/N) = o(x(t))$ for all $t \le t^*$, and so $\Pr (T_1 \le t^* \wedge T_0) = 0$, provided
$N$ is sufficiently large.
This completes the proof of Lemma~\ref{lem.slow}.

\section{Initial phase: upper bounds} \label{sec:initial}

In this section, we show that, if $X_N (0) >  (\mu-\lambda)N \omega(N)$, then by the time
$t_0 = 1/(\omega(N) \lambda (\mu-\lambda))$, $X_N(t)$
with high probability will have dropped down below $(\mu-\lambda)N \omega(N)$.

To this end, we give a lemma showing that $\E (X_N (t)/N)$ is always bounded above by the
solution $x(t)$ of the differential equation~(\ref{eq.diff-eq}).  This result has earlier been proved by Allen (2008, p94), and in a more general
setting by Simon and Kiss (2013).

\begin{lemma} \label{lem.inequality}
Suppose that $X_N(0) \in \{0,1, \ldots, N\}$ is non-random,
and let $x(t)$ be the solution to~(\ref{eq.diff-eq}) with initial condition
$x(0) = X_N(0)/N$.  Then, for all $t \ge 0$,
$$
\E  X_N (t) \le N x(t).
$$
\end{lemma}

\begin{proof}  (Sketch)
It is easy to calculate that, for all $t \ge 0$,
\begin{eqnarray*}
\ddt{} \big( \E X_N (t) - N x(t) \big)
&=& \big( \E X_N (t) - N x(t) \big) \left\{ - (\mu-\lambda) - \lambda \Big(\frac{\E X_N (t)}{N} + x(t)\Big) \right\} \\
&&\mbox{} - \frac{\lambda}{N} \E \big(X_N (t) - \E X_N (t)\big)^2.
\end{eqnarray*}
Using the integrating factor $\exp\Big( t (\mu - \lambda) + \lambda  \int_{0}^t (N^{-1} \E X_N (s) + x(s))\, ds \Big)$,
and the fact that $\E X_N (0) - N x(0) = 0$, it follows that $Y_N (t) = \E X_N (t) - N x(t)$ satisfies
$$
Y_N(t) = - \frac{\lambda}{N} \int_{s=0}^t \E (X_N (s) - \E X_N (s))^2 e^{ - (t-s) (\mu - \lambda)
- \lambda  \int_s^t (N^{-1} \E X_N (u) + x(u)) \, du}  \, ds,
$$
and so is non-positive for all $t\ge 0$.
\end{proof}

The previous result implies the following lemma.

\begin{lemma} \label{lem.quick}
Let $\omega(N)$ be any function tending to infinity, and set $t_0 = t_0(N) = \frac{1}{\omega(N)^{1/2} \lambda (\mu-\lambda)}$. Then, for any initial state $X_N(0)$,
$$
\Pr \big(X_N (t_0) \ge N (\mu-\lambda) \omega(N) \big) = o(1).
$$
\end{lemma}

\begin{proof}
We note that, for any value of $x(0)$, and any $t \ge 0$,
$$
x(t) = \frac{ x(0) (\mu-\lambda) e^{- (\mu-\lambda)t}}{(\mu-\lambda) + x(0) \lambda (1-e^{-(\mu-\lambda)t})}
\le \frac{(\mu-\lambda) e^{-(\mu-\lambda)t}}{\lambda (1- e^{-(\mu-\lambda)t})}
\le \frac{1}{\lambda t}.
$$
Here we used the inequality $e^{-u} \le (1-e^{-u})/u$, valid for all $u > 0$.

Therefore we have $x(t_0) \le (\mu-\lambda) \omega(N)^{1/2}$, for any value of $x(0)$.
Hence, by Lemma~\ref{lem.inequality}, we have $\E X_N (t_0) \le (\mu-\lambda) N \omega(N)^{1/2}$, for any initial state $X_N(0)$,
and it follows that
$$
\Pr \Big(X_N (t_0) \ge (\mu-\lambda)N \omega(N) \Big) \le \frac{1}{\omega(N)^{1/2}} = o(1).
$$
\end{proof}

\section{Proof of Theorem~\ref{thm.main}} \label{sec.proof}

In this section, we assemble the lemmas from the preceding three sections into a proof of Theorem~\ref{thm.main}.

Given two copies $(B(t))$ and $(\tilde{B}(t))$ of a
continuous-time birth-and-death chain, with $B(0) \le \tilde{B}(0)$,
we can couple them in a monotone way.
If $B(t) = \tilde{B}(t)$, then they make the next jump (and all subsequent jumps) together.
For as long as $B(t) \not = \tilde{B}(t)$, $(B(t))$ and $(\tilde{B}(t))$ evolve independently, so that a.s.\ they do not jump simultaneously. This ensures that a.s. the two copies of the chain
never cross, so that $B(t) \le \tilde{B}(t)$ a.s. for all $t$.

\begin{proof}
Recall from (\ref{eq:omega}) that $\omega(N) = \big( N^{1/2} (\mu-\lambda) \big)^{1/4}$.
We distinguish three ranges for the starting state $X_N(0)$, assuming always that $(\mu-\lambda)X_N(0) \to \infty$:
\begin{enumerate}
\item[(a)] $X_N(0) \le 2N^{1/2}\omega(N)$,
\item[(b)] $2N^{1/2} \omega(N) < X_N(0) \le (\mu-\lambda)N \omega(N)$,
\item[(c)] $X_N(0) > (\mu - \lambda)N \omega(N)$.
\end{enumerate}
It could be that $X_N(0)$ falls into different ranges for different values of $N$: we partition the set of natural numbers into three sets depending
on which of (a), (b), (c) holds.  It suffices to prove the result separately for whichever subsequence(s) are infinite, and so we may treat each of
the three ranges in turn, working (tacitly) with an infinite sequence of values of $N$ for which the inequalities defining the range hold.

\medskip
\noindent
(a) Suppose $(\mu-\lambda) X_N(0) \to \infty$ and $X_N(0) \le 2N^{1/2} \omega(N)$.  By
Lemma~\ref{lem.late}, 
$$
(\mu - \lambda) T_N - \big( \log X_N(0) + \log (\mu-\lambda) - \log \mu \big) \to W,
$$
in distribution, as $N \to \infty$, where $W$ has the standard Gumbel distribution.  This is~(\ref{eq.low}), which is equivalent to~(\ref{eq.general})
in this range.

\medskip

\noindent
(b) Let $X^*= X^*(N) = N^{1/2} \omega(N)$.
Suppose that $2N^{1/2} \omega (N) < X_N(0) \le (\mu-\lambda)N \omega(N)$.  We run $(X_N(t))$ for a time
$t^* = t_{X_N(0)/N}(X^*/N)$, as in the statement of Lemma~\ref{lem.slow}.
Let $E$ be the event that
$X^*(1 - \omega(N)^{-1/3}) \le X_N(t^*) \le X^*(1 + \omega(N)^{-1/3})$; by Lemma~\ref{lem.slow}, $\Pr(\overline E) = o(1)$.

Let $(Y_N(t))$ be a copy of the logistic process with $Y_N(0)= X^*(1 - \omega(N)^{-1/3})$, and $(Z_N(t))$ be a copy
with $Z_N(0) = X^*(1 + \omega(N)^{-1/3})$.  On the event $E$, we couple these with $(X_N(t))$ from time $t^*$ onwards in a monotone way, so that $Y_N(t) \le X_N(t^*+t) \le Z_N(t)$ for all $t \ge 0$.  Then, on the event~$E$,
$T^{Y_N} + t^* \le T_N \le T^{Z_N} + t^*$.  By Lemma~\ref{lem.late}, as $N \to \infty$, in distribution,
$$
(\mu - \lambda) T^{Y_N} - \big( \log X_N(0) + \log (1- \omega(N)^{-1/3}) + \log (\mu-\lambda) - \log \mu \big) \to W,
$$
where $W$ is a standard Gumbel random variable.  Since $\log(1-\omega(N)^{-1/3}) = o(1)$, we then have from the asymptotic formula (\ref{eq.number}) for $t^*$, as $N \to \infty$,
$$
(\mu - \lambda) (T^{Y_N} + t^*) - \Big( \log X_N(0) + \log (\mu-\lambda) - \log \Big( 1 + \frac{\lambda X_N(0)}{(\mu-\lambda)N}\Big) - \log \mu \Big) \to W,
$$
in distribution, and the same holds when $T^{Y_N}$ is replaced by $T^{Z_N}$.  Hence, as $N \to \infty$,
$$
(\mu - \lambda) T_N - \Big( \log X_N(0) + \log (\mu-\lambda) - \log \Big( 1 + \frac{\lambda X_N(0)}{(\mu-\lambda)N}\Big) - \log \mu \Big) \to W,
$$
in distribution.  This is (\ref{eq.equiv}), which we have seen is equivalent to~(\ref{eq.general}).

\medskip

\noindent
(c) Suppose $X_N(0) > (\mu - \lambda)N \omega(N)$. Let $\kappa_N$ be the hitting time of $\lfloor (\mu - \lambda)N \omega (N) \rfloor$. Let $t_0 = 1/\omega(N)^{1/2}\lambda (\mu-\lambda)$,
as in Lemma~\ref{lem.quick}; by Lemma~\ref{lem.quick}, $\kappa_N \le t_0$ with probability $1-o(1)$. Then $T_N$ is the sum of $\kappa_N$ and the time to extinction from state $\lfloor (\mu - \lambda)N \omega (N) \rfloor$. So $T_N$ is bounded below by
$$\frac{1}{\mu - \lambda} \big( \log \big((\mu-\lambda)N \omega(N)\big) + \log (\mu-\lambda) - \log \big( 1 + \lambda \omega(N)\big) - \log \mu  + W_N \big ),$$
and, with probability $1-o(1)$, bounded above by
$$\frac{1}{\mu - \lambda} \big(t_0 (\mu-\lambda) + \log \big((\mu-\lambda)N \omega(N)\big) + \log (\mu-\lambda) - \log \big( 1 + \lambda \omega(N)\big) - \log \mu  + W_N \big ),$$
where $W_N$ converges in distribution to a standard Gumbel random variable $W$.
%
%
Since $(\mu-\lambda) t_0 = o(1)$ and $\omega (N) \to \infty$, it follows that, as $N \to \infty$,
$$
(\mu - \lambda) T_N - \big( \log N + 2 \log (\mu-\lambda) - \log \lambda - \log \mu \big) \to W,
$$
in distribution.  This is~(\ref{eq.high}), which is equivalent to~(\ref{eq.general}) in this range.

\smallskip

This completes the proof.
\end{proof}

\section{The critical regime}  \label{sec.critical}

Our methods can also be applied in the critical regime, where $|\mu-\lambda| = O(N^{-1/2})$.
In this case, there exist constants $\delta, c > 0$ (depending on $\limsup_{N\to \infty} (\lambda-\mu)N^{1/2}$)), such that, regardless of the value taken by $X_N(0)$,
$\Pr (T_N \le c N^{1/2}) > \delta$.  One way to prove this is as follows: (a)~apply Lemma~\ref{lem.inequality} for
$t = N^{1/2}$, and so $x(t) \le 1/(\lambda t)= O(N^{-1/2})$, to show that, for some constant $c_1$, with positive probability, $X_N(t) \le c_1 N^{1/2}$, uniformly
in $X_N(0)$; (b)~compare $(X_N(t))$ with a linear birth-and-death chain with the same parameters, with initial state $c_1 N^{1/2}$,
and show that, for some constant $c_2$, with positive probability, $(X_{N}(t))$ reaches~0 in a further time $c_2 N^{1/2}$.  Thus there is a
positive probability of extinction by time $(c_1+c_2)N^{1/2}$, whatever the initial state.  It now follows, by
repeated trials, that $\Pr (T_N > \omega(N) N^{1/2}) \to 0$ whenever $\omega(N) \to \infty$.  Throughout the critical regime, a lower bound
on the extinction time of the form $\Pr (T_N \le \eps(N) N^{1/2} ) \to 0$ whenever $\eps(N)\to 0$ can again be obtained by
comparing with a suitable linear birth-and-death chain.  Much more precise results concerning the process in the critical regime with initial state
of order $N^{1/2}$ are given by Dolgoarshinnykh and Lalley~(2006).
In this regime, both the expected extinction time and the fluctuations are of order $N^{1/2}$, so that we do not have cut-off.  This is in line with
our results for the barely subcritical regime showing that cut-off becomes less pronounced as we approach the critical regime from below.

\section{Total number of cases}
\label{sec.total}

We now turn our attention to the total number $C_N$ of new cases (infection events) from the start of the epidemic until its extinction.

We will prove that, provided $X_N(0)(\mu-\lambda) \to \infty$, the total number
$C_N$ of cases is concentrated around its expectation, which is close to
$$
N \frac{\mu}{\lambda} \log \left( 1 + \frac{\lambda X_N(0)}{N (\mu - \lambda)}  \right) - X_N(0).
$$
In the case where $X_N(0)/N(\mu-\lambda) \to 0$, our results imply that the expectation of $C_N$ is close to $\frac{\lambda X_N(0)}{\mu-\lambda}$
and the variance is of the order at most $X_N(0)(\mu-\lambda)^{-3}$.  (This can be interpreted as saying that the epidemic behaves as
$X_N(0)$ independent outbreaks from a single initial infective.)  If $X_N(0)/N(\mu - \lambda) \to \infty$, then
our results show that $C_N$ has expectation approximately $N \log\Big( X_N(0) / N (\mu-\lambda)\Big)$,
and variance of order at most $N(\mu-\lambda)^{-2}$.

\begin{theorem}
\label{thm.total}
Suppose that $\mu = \mu(N)$ and $\lambda= \lambda(N)$ are bounded away from both~0 and infinity.
Suppose also that $(\mu - \lambda)N^{1/2} \to \infty$ as $N \to \infty$, and that $X_N(0)$ is non-random.
Let $v_N(x) = \min \Big ( \frac{x^{1/2}}{(\mu-\lambda)^{3/2}} , \frac{N^{1/2}}{\mu-\lambda} \Big)$. Then, for any $\eps > 0$, there exists $K(\eps)$ such that, for $N$ sufficiently large,
\begin{eqnarray*}
\Pr \Big (\Big |C_N  - \frac{\mu}{\lambda} N \log \Big ( 1 + \frac{\lambda X_N(0)}{N (\mu - \lambda)}  \Big ) + X_N(0)\Big | \ge K(\eps) v_N(X_N(0))\Big ) \le \eps.
\end{eqnarray*}
\end{theorem}
\begin{proof}
For $X=0, \dots, N$, let $\ell_X = \lambda(1-X/N)$, so that the birth rate when $X_N(t) = X$ is equal to $X \ell_X$.
Then $C_N$ has the same distribution as the number of births before extinction in a corresponding discrete-time
birth-and-death chain $(\hat{X}_N(t))$, where the probability of a birth in state~$X$ is $\ell_X/(\mu+\ell_X)$, and the probability of a death is
$\mu/(\mu + \ell_X)$, so we can work with $(\hat{X}_N(t))$ instead.

Given $\hat{X}_N (0) = X$ , the number of births before extinction can be represented as the sum of independent random
variables $C_{N,Y}$, for $Y = X, X-1, \dots, 1$, where $C_{N,Y}$ is the number of new cases starting in state~$Y$ until
hitting state~$Y-1$.
We will show that, for all $Y = 1, \ldots, N$,
\begin{eqnarray*}
\left( 1 + \frac{\mu \lambda }{N (\mu-\ell_Y)^2}  \right)^{-1} \frac{\ell_Y}{\mu - \ell_Y} \le \E C_{N,Y} \le \frac{\ell_Y}{\mu - \ell_Y}.
\end{eqnarray*}
Conditioning on the first step in a standard way, we see that
$\mu \E C_{N,Y} = \ell_Y (\E C_{N,Y+1} + 1)$, for $1 \le Y \le N-1$.
We now proceed by downward induction.
Since $\ell_N = 0$, both our upper and lower bounds on $\E C_{N,N}$ are equal to zero, which is the true value.
Suppose that we have the stated upper bound on $\E C_{N,Y+1}$, so that, using the fact that $\mu \ge \lambda$ implies $\mu - \ell_Z \ge 0$ for all $Z$,
$\E C_{N,Y+1} \le \frac{\ell_{Y+1}}{\mu - \ell_{Y+1}} \le \frac{\ell_Y}{\mu - \ell_Y}$.  Then
$$
\E C_{N,Y} = \frac{\ell_Y}{\mu} (\E C_{N,Y+1} + 1) \le \frac{\ell_Y}{\mu} \left( \frac{\ell_Y}{\mu -\ell_Y} + 1 \right) = \frac{\ell_Y}{\mu-\ell_Y},
$$
which is the required upper bound on $\E C_{N,Y}$.

Suppose now that we have the stated lower bound on $\E C_{N,Y+1}$, so that
$$
\E C_{N,Y+1} \ge \left(1 + \frac{\mu \lambda }{N (\mu-\ell_{Y+1})^2}  \right)^{-1} \frac{\ell_{Y+1}}{\mu - \ell_{Y+1}}.
$$
Then, since $\ell_Y \ge \ell_{Y+1}$,
\begin{eqnarray*}
\E C_{N,Y+1} + 1 &\ge&
\frac{\mu}{\mu-\ell_{Y+1}} - \frac{\ell_{Y+1}}{\mu - \ell_{Y+1}}\left( 1 - \frac{1}{ 1 + \frac{\mu \lambda }{N (\mu-\ell_{Y+1})^2} } \right) \\
&\ge& \frac{\mu}{\mu-\ell_{Y+1}} - \frac{\ell_Y}{\mu - \ell_{Y+1}}\left( 1 - \frac{1}{ 1 + \frac{\mu \lambda }{N (\mu-\ell_Y)^2} } \right) \\
&=& \frac{\mu}{\mu-\ell_{Y+1}} - \frac{\ell_Y}{\mu - \ell_{Y+1}}
\frac{\frac{\mu \lambda }{N (\mu-\ell_Y)^2} }{\left( 1 + \frac{\mu \lambda }{N (\mu-\ell_Y)^2} \right)} \\
&=& \frac{\mu}{\mu-\ell_{Y+1}} \left(\frac{1 + \frac{\mu \lambda }{N (\mu-\ell_Y)^2} - \frac{\lambda \ell_Y}{N(\mu-\ell_Y)^2}}{1 + \frac{\mu \lambda }{N (\mu-\ell_Y)^2}}\right) \\
&=& \frac{\mu}{\mu-\ell_{Y+1}} \left(\frac{1 + \frac{\lambda }{N (\mu-\ell_Y)}}{1 + \frac{\mu \lambda }{N (\mu-\ell_Y)^2}}\right)
= \frac{\mu}{\mu-\ell_Y} \frac{1}{\left(1 + \frac{\mu \lambda }{N (\mu-\ell_Y)^2}\right)}.
\end{eqnarray*}
In the last step, we also used that $\mu - \ell_{Y+1} = \mu - \ell_Y + \lambda/N$,
and so $1 + \frac{\lambda }{N (\mu-\ell_Y)} = \frac{\mu-\ell_{Y+1}}{\mu-\ell_Y}$. It follows, as required for the induction step, that
\begin{eqnarray*}
\E C_{N,Y} &=& \frac{\ell_Y}{\mu} (\E C_{N,Y+1} + 1)
\ge \frac{\ell_Y}{\mu-\ell_Y} \left(1 + \frac{\mu \lambda }{N (\mu-\ell_Y)^2}  \right)^{-1}.
\end{eqnarray*}

\smallskip

It follows from the above bounds that
\begin{eqnarray*}
\E C_N &\le& \sum_{Y=1}^{X_N(0)} \frac{\ell_Y}{\mu - \ell_Y} \le \int_{x=0}^{X_N(0)} \left( \frac{\mu}{\mu - \lambda(1-x/N)} - 1 \right) \, dx \\
&=& \frac{\mu N}{\lambda}\log \left( 1 + \frac{\lambda X_N(0)}{N(\mu-\lambda)} \right) - X_N(0),
\end{eqnarray*}
and, since $N(\mu-\lambda)^2 \to \infty$, also that, for $N$ sufficiently large,
$$
\E C_N \ge \sum_{Y=1}^{X_N(0)} \frac{\ell_Y}{\mu - \ell_Y} - \frac{2\mu \lambda^2}{N} \sum_{Y=1}^{X_N(0)} \frac{1}{(\mu-\ell_Y)^3}.
$$
Noting that
\begin{eqnarray*}
\lefteqn{\int_{x=0}^{X_N(0)} \left( \frac{\mu}{\mu - \lambda(1-x/N)} - 1 \right) \, dx } \\
&\le& \sum_{Y=0}^{X_N(0)-1} \frac{\ell_Y}{\mu - \ell_Y}
= \sum_{Y=1}^{X_N(0)} \frac{\ell_Y}{\mu - \ell_Y} + \frac{\lambda}{\mu-\lambda} - \frac{\ell_{X_N(0)}}{\mu-\ell_{X_N(0)}} \\
&\le& \sum_{Y=1}^{X_N(0)} \frac{\ell_Y}{\mu - \ell_Y} + \frac{\mu(\lambda - \ell_{X_N(0)})}{(\mu-\lambda)(\mu-\ell_{X_N(0)})} \\
& \le & \sum_{Y=1}^{X_N(0)} \frac{\ell_Y}{\mu - \ell_Y} + \frac{\mu\lambda X_N(0)}{N(\mu-\lambda)^2},
\end{eqnarray*}
that
\begin{eqnarray*}
\sum_{k=1}^X \frac{1}{(\mu-\ell_k)^3} &\le&
\min \left \{\frac{X}{(\mu-\lambda)^3}, \int_{x=0}^\infty \frac{1}{(\mu-\lambda + \lambda x / N)^3} \, dx  \right \} \\
&=& \min \left \{ \frac{X}{(\mu-\lambda)^3}, \frac{N}{2\lambda (\mu-\lambda)^2}  \right \},
\end{eqnarray*}
and that $X_N(0) \le N$, we see that, for $N$ large enough,
\begin{eqnarray}
\lefteqn{\left| \E C_N - \left( \frac{\mu N}{\lambda}\log \left( 1 + \frac{\lambda X_N(0)}{N (\mu-\lambda)} \right) - X_N(0) \right) \right| } \nonumber \\
&\le& \frac{\mu \lambda X_N(0)}{N(\mu-\lambda)^2} + \min \left( \frac{2\mu \lambda^2 X_N(0)}{(\mu-\lambda)^3 N} , \frac{\mu \lambda}{(\mu-\lambda)^2}\right) \nonumber \\
&\le& 3\min \left( \frac{\mu^2 \lambda X_N(0)}{(\mu-\lambda)^3 N} , \frac{\mu \lambda}{(\mu-\lambda)^2}\right). \label{tc-mean}
\end{eqnarray}

\smallskip

We now estimate the variance of $C_N$, noting that $\Var C_N = \sum_{Y=1}^{X_N(0)} \Var C_{N,Y}$.

Starting from $Y$ and until the hitting time $\tau_{N,Y}$ of $Y-1$ by $(\hat{X}_N(t))$, we can couple $(\hat{X}_N(t))$ with a discrete chain $(\hat{X}_{N,Y}(t))$ where, in any state, the probability of a birth is $\ell_Y /(\mu+\ell_Y)$ and the probability of a death is $\mu/(\mu + \ell_Y)$, in such a way that $\hat{X}_N (t) \le \hat{X}_{N,Y}(t)$ for $0 \le t \le \tau_{N,Y}$.
Letting $D_{N,Y}$ be the number of births in $(\hat{X}_{N,Y}(t))$ starting from $Y$ until hitting $Y-1$, we thus see that, under the coupling, $C_{N,Y} \le D_{N,Y}$. It follows that
\begin{eqnarray*}
\Var C_{N,Y} &=& \E C_{N,Y}^2 - (\E C_{N,Y})^2 \le \E D_{N,Y}^2 = \Var D_{N,Y} + (\E D_{N,Y})^2.
\end{eqnarray*}
For a given value of $Y$, consider a discrete random walk, starting at $1$, with probability $p = \mu/(\ell_Y + \mu)$ of a down-step and probability
$q = \ell_Y/(\ell_Y + \mu)$ of an up-step. Then $D_{N,Y}$ has the same distribution as $(T_{N,Y}-1)/2$, where $T_{N,Y}$ is the hitting time of the origin for this walk. Standard arguments imply that the generating function $G_{N,Y}$ of $T_{N,Y}$ satisfies the recurrence
$G_{N,Y} (z) = p z + q z (G_{N,Y} (z))^2$, and so $G_{N,Y}(z) = (1 - \sqrt{1-4pqz^2})/2qz$. Differentiating, we obtain $\E D_{N,Y} = \ell_Y /(\mu-\ell_Y)$ and
$\displaystyle \Var D_{N,Y} = \ell_Y \mu (\ell_Y + \mu) / (\mu-\ell_Y)^3$, and hence $\displaystyle \Var C_{N,Y} \le \frac{2\lambda \mu^2}{(\mu-\ell_Y)^3}$. Summing over $Y$,
\begin{eqnarray}
\label{tc-var}
\Var C_N \le 2 \mu^2 \min  \left( \frac{\lambda X_N(0)}{(\mu-\lambda)^3}, \frac{N}{(\mu-\lambda)^2} \right).
\end{eqnarray}

Suppose that $X_N(0)/N(\mu - \lambda) \to \infty$. Then for $N$ large enough, the upper bound in~\eqref{tc-mean} is equal to $\frac{3\mu \lambda}{(\mu-\lambda)^2} \le N^{1/2}/(\mu - \lambda)$, and $\Var C_N \le 2\mu^2 N/(\mu - \lambda)^2$ in this case. If $X_N(0)/N(\mu - \lambda)$ is bounded, then the upper bound in~\eqref{tc-mean} is at most of the order $\displaystyle \frac{X_N(0)}{(\mu-\lambda)^3 N}$, and, for $N$ large enough,
$$\frac{X_N(0)^{1/2}}{(\mu - \lambda)^{1/2} N^{1/2}} \cdot \frac{1}{(\mu - \lambda) N^{1/2}} \cdot \frac{X_N(0)^{1/2}}{(\mu - \lambda)^{3/2}}\le \frac{X_N(0)^{1/2}}{(\mu - \lambda)^{3/2}},$$
while $\Var C_N$ is of the order at most $\displaystyle \frac{\lambda X_N(0)}{(\mu-\lambda)^3}$. In both cases, the theorem follows by Chebyshev's inequality.
\end{proof}

\section{Numerical methods}
  \label{sec.finite}

The stochastic SIS logistic process $(X_N(t))$, as defined by events and
rates~(\ref{eq:eventsrates}), can be analysed through use of its corresponding Kolmogorov forward
equations, as we now explain.  We will not index all terms by $N$ explicitly here for notational
simplicity, but the method of analysis is for a population of size $N$.  We
start by writing $p_X(t) = \Pr(X_N(t)=X)$, and let $\myvec{p}(t)$ be a column
vector whose $X$-th entry is $p_X(t)$ -- our convention is that such a vector
starts at its $0$-th element and has length $N+1$.  (We follow the more applied literature in treating
$\myvec{p}(t)$ as a column vector.)  The Kolmogorov forward equations
then take the form of a linear system of differential equations
\ba
\ddt{p_X} & = - \left(\mu X +\lambda X\left(1-\frac{X}{N}\right)\right) p_X +
\lambda (X-1)\left(1-\frac{X-1}{N}\right)
p_{X-1} \\ & \qquad + \mu (X+1) p_{X+1} \text{ ,} \qquad 0<X<N \text{ ,} \\
\ddt{p_0} & = \mu p_1 \text{ ,} \\
\ddt{p_N} & = -\mu N p_N + \lambda \frac{N-1}{N}p_{N-1} \text{ .}
\label{compkol}
\ea
These can be expressed in the form
\be
\ddt{\myvec{p}} = \mymat{M}\myvec{p} \text{ ,} \label{matkol}
\ee
where $\mymat{M}$ is an $(N+1)\times(N+1)$ matrix.  Quantities of interest
include:
\be
\E X_N(t) = \myvec{X}\cdot\myvec{p}(t) \text{ ,}\quad
F_{T_N}(t) = \Pr(T_N \le t) = p_0(t) \text{ ,}\quad
f_{T_N}(t) = \mu p_1(t) \text{ ,}
\ee
where $\myvec{X}$ is a vector whose $X$-th element is $X$, $F_{T_N}$ is the
distribution function of the extinction time and $f_{T_N}$ is the probability
density function of the extinction time, which has the form above due to the
second equation in~\eqref{compkol}.  To integrate~\eqref{matkol} numerically, we make use of
the implicit Euler scheme, as has been advocated for stochastic epidemic models
by Jenkinson and Goutsias~(2012). We now sketch the arguments and approach
presented in that paper.  First, note that the solution of~\eqref{matkol} is
given by a matrix exponential,
\be
\myvec{p}(t) = \Exp(\mymat{M} t) \myvec{p}(0) \text{ ,} \nonumber
\ee
and therefore over a time interval $[t,t+h]$ we can write
\ba
\myvec{p}(t+h) & = \Exp(\mymat{M} h) \myvec{p}(t) =
(\mymat{I} + \mymat{M} h + O(h^2)) \myvec{p}(t) \text{ ,} \\
\myvec{p}(t) & = \Exp(-\mymat{M} h) \myvec{p}(t+h) =
(\mymat{I} - \mymat{M} h + O(h^2)) \myvec{p}(t+h) \text{ ,}
\label{matrixexpand}
\ea
where $\mymat{I}$ is the $(N+1)\times(N+1)$ identity matrix. The implicit Euler
numerical scheme is based on the second equation in~\eqref{matrixexpand} and
involves solving the matrix equation
\be
(\mymat{I} - \mymat{M} h) \myvec{p}_+ = \myvec{p}
\label{ie}
\ee
to obtain an approximation $\myvec{p}_+$ to $\myvec{p}(t+h)$, in terms of an approximation $\myvec{p}$ to $\myvec{p}(t)$, at each timestep,
for example by using \textsc{Matlab}'s \verb|\| operator.  We see from the first equation of~\eqref{matrixexpand} that
the error introduced at each timestep is $O(h^2)$, and so the global error over the interval $[0,t]$ is $O(th)$ as $h \to 0$.
This means that, in practice, for given choices of $N$, $X_N(0)$, $\lambda$ and $\mu$, we can tune $h$ to
achieve any desired accuracy.  Suppose that $\mathbf{p}$ is a probability vector; we now show that
$\mathbf{p}_+$ generated by~\eqref{ie} is also a probability vector.
First, premultiplying~\eqref{ie} by a row vector of ones, $\mathbf{1}^{\top}$, we obtain
\be
1 = \mathbf{1}^{\top} \myvec{p} = \mathbf{1}^{\top}\mymat{I} \myvec{p}_+
+ h \mathbf{1}^{\top} \mymat{M} \myvec{p}_+ =\mathbf{1}^{\top} \myvec{p}_+\text{ ,}
\ee
which holds because $\mymat{M}$ generates a Markov chain and so its columns must
sum to zero: $\mathbf{1}^{\top} \mymat{M} = \mathbf{0}$.
Secondly, from its definition, the off-diagonal elements of $\mymat{M}$
are non-negative, meaning that the off-diagonal elements of $\mymat{I} -
\mymat{M} h$ are non-positive and so, after checking for non-singularity,
all elements of $(\mymat{I} - \mymat{M} h)^{-1}$ are non-negative.

Often, solution of equations such as~\eqref{matkol} for $\myvec{p}(t)$ is more
numerically efficient (particularly for calculating distributions of quantities
like extinction times) than Monte Carlo methods that use (pseudo-)random number
generation; see Keeling and Ross (2008).  To see why this should be so for our case, note that one
of the quantities we wish to calculate is the probability distribution function
for the extinction time of an epidemic with rates $\lambda$ and $\mu$ of
order~1 with $\mu - \lambda \approx 10^{-3}$, population size $N=10^7$, and
$X_N(0) = N$.  For these parameter values, our asymptotic results from
Theorem~\ref{thm.total} give that we expect to see over $10^8$ events. Using
Monte Carlo methods, we would need to simulate each of these to achieve one
extinction, and would need to simulate many realisations of the entire epidemic
to control the Monte Carlo error. In contrast, use of the forward Euler method
as above requires just one realisation with the step-size $h$ at an appropriate
value relative to required numerical error, and solution of an $N$-dimensional
sparse linear system at each step.

We performed a comparison of the simulation results based on implicit Euler
solution of the Kolmogorov forward equations with our asymptotic results for a range of
values of $N$ from $10^1$ to $10^7$, keeping $\mu=1$ throughout, for two
different scenarios. In the first scenario, we do not scale $\mathcal{R}_0 = \lambda$
with $N$, but instead leave it constant at $0.9$.  In the second scenario, we
scale $1-\mathcal{R}_0$ approximately like $N^{-1/3}$, starting with $0.9$ for
$N=10$.  The results for these two scenarios are pictured in
Figures~\ref{fig:unscaled} and~\ref{fig:scaled} respectively.  These demonstrate
that the asymptotic results can be a good approximation to the system behaviour
for large population sizes (e.g., on the scale of a town, city or country) with
regard to the extinction times, and, for any population of more than a hundred,
also for the mean number of infectives.  Furthermore, they show that for the
unscaled case, extinctions happen relatively quickly for all values of $N$, but
that, as $\cR_0$ tends to 1 with $N$, extinctions can take an extremely long
time to occur despite an initial fast decline in $\mathbb{E} X_N (t)$.

We also give numerical results for the total number of cases $C_N$. For this,
we adapt the path sum numerical method introduced by Ross~(2011); we shall show
that this is suited to rapid calculation not only of the probability mass
function (as in Ross (2011)) but also of the mean variance of $C_N$ (see
\eqref{psmat} and \eqref{8new} below).  To implement this method, we track
the current state as $X_N \times B \in \{ 0, \dots, N\}\times \{0,1\}$, where
the Bernoulli random variable $B$ is defined to be $1$ if the last event was an
infection, and $0$ if the last event was a recovery.  Since we are interested
in the final number of cases, we only need to consider the jump chain for this
process. We will start the system in state $(X_{N,0}, B_0)$.
Writing the state after $u$ events as $(X_{N,u}, B_u)$, we see that
\begin{equation}
C_N = \sum_{u=0}^{\infty} B_u \text{ .}
\end{equation}
(For our model, it is possible to recover $C_N$ from the
total number of events without the need for the auxiliary variable $B$; in more complex
epidemic models, the auxiliary variable aids easy calculation of the quantities of interest.)

The transition probabilities for the jump chain are:
\ba
\mathbb{P}(X_{N,u+1} = X+1, B_{u+1} =1 | X_{N,u} = X, B_u) &=
\begin{cases}
	\frac{\lambda (1-X/N)}{\lambda(1-X/N) + \mu} & \text{if } X>0 \text{,} \\
	0 & \text{otherwise.}
\end{cases} \\
\mathbb{P}(X_{N,u+1} = X-1, B_{u+1} =0 | X_{N,u} = X, B_u) &=
\begin{cases}
	\frac{\mu}{\lambda(1-X/N) + \mu} & \text{if } X>0 \text{,} \\
	0 & \text{otherwise.}
	\label{jumpeqns}
\end{cases}
\ea
Clearly, the state space decomposes into an absorbing
class $\{(0,0), (0,1)\}$ and a transient class $\mathcal{T}$ of states with a positive number of infectives.
For a state $i = (X,B)$, we set $b_i$ equal to $B$.  We write $P_{i,j}$ for the probability of
moving from state $i$ to state $j$ as defined in~\eqref{jumpeqns}. We write
$c_i$ for the expected value of the random variable $C_N$ given
that the initial state is~$i$.  A standard calculation conditioning on the first step then shows that, for each $i \in \mathcal{T}$,
$c_i = \sum_j P_{ij} (c_j + b_j)$.
Let $g_i = c_i + b_i$; then, for each $i$, $g_i = b_i + \sum_j P_{ij} g_j$. It follows that
\begin{equation}
	\mygvec{g} = (\mymat{I} - \mymat{P})^{-1} \myvec{b} \text{ .}
	\label{psmat}
\end{equation}
We note that $g_i$ differs from $c_i$ by at most 1, and only in the case where $b_i =1$.
Note further that: (i)~the inverse in this equation does not need to be calculated
explicitly, and instead a system of linear equations can be solved, for example
using the backslash operator \verb|\| in \textsc{Matlab}, and (ii) we have
restricted attention to the transient states so that the inverse
in~\eqref{psmat} is well-defined.

To study the variability of the distribution of
$C_N$ about its mean, we let $h_i = \mathbb{E} [ C_N(C_N-1) \mid X_N(0) = i]$.
From Ross (2011), if $\phi_i(z) := \E [z^{C_N} | X_N(0) = i]$, then
\begin{equation} 
	\phi_i(z) = z^{b_i} \sum_{j} P_{i,j} \phi_j(z) \qquad (i \in \mathcal{T}) \text{ .}
\end{equation}
We note that $h_i = \phi_i''(1)$ and hence
\begin{equation} 
\sum_{j\in \mathcal{T}} (\delta_{i,j} - P_{i,j}) h_j = 2 b_i
\sum_{j\in \mathcal{T}} P_{i,j} c_j
\qquad (i \in \mathcal{T}) \text{ .}
	\label{8new}
\end{equation}
This equation for $h_i$ can also be evaluated by solving a system of linear
equations and used to calculate the standard deviation of $C_N$.

The results of comparing the path sum to the asymptotic formula for the mean,
as well as the asymptotic variance bound~\eqref{tc-var}, using the same
parameter choices as previously, are shown in Figure~\ref{fig:cns}.  These
results exhibit rapid convergence of the (scaled) mean to its asymptotic value
as the population size $N$ gets large, and also rapid reduction in the
variability of the distribution.

\section{The relationship between the deterministic process and time of extinction}

\label{sec.mean-extinction}

One feature that becomes apparent by studying the numerical results is that,
especially in the barely subcritical regime, there is a clear distinction
between the time that the size of the epidemic first becomes ``small'' (which
is in practice often taken as a proxy for the end of the epidemic) and the time
that extinction occurs with high probability.

In a situation where control measures have brought an epidemic into a subcritical regime, but observations of the
prevalence of the epidemic are only partial, it is potentially important to infer the likely time of extinction from the
existing observations and/or fits to models governed by differential equations, so that control measures can be maintained
for long enough that the epidemic has died out with high probability.

For the SIS logistic process, our results indicate that an appropriate ``guide time'' to the extinction time is
the time $\widehat{t}$ at which the deterministic process, given by (\ref{eq.solution}) and starting from $X_N(0)/N$, reaches $\widehat{x} = \frac{\mu}{(\mu-\lambda)N}$ (i.e., when the number of infectives
is projected to be $\mu/(\mu-\lambda)$).  Note that $\widehat{t}$ occurs significantly later than the time when the deterministic process reaches
$\varepsilon N$, for $\varepsilon$ a small constant, but (in the near-critical regime) considerably earlier than the time when the
deterministic process reaches $1/N$, corresponding to a single remaining infective.

To see that the time $\widehat{t}$ has the desired property, we re-write our result (\ref{eq.general}) in terms of the function $s(x)$
introduced in (\ref{eq.ST}).  We note that
\begin{eqnarray*}
s(X_N(0)/N) &=& \log \big( 1 + \frac{\lambda}{\mu-\lambda} \frac{X_N(0)}{N}\big) - \log (X_N(0)/N) \\
&=& \log\big( 1 + (\mu-\lambda)N/\lambda X_N(0) \big) - \log (\mu-\lambda) + \log \lambda,
\end{eqnarray*}
and deduce that
$$
(\mu-\lambda)T_N - \Big( s(\hat{x}) - s(X_N(0)/N) - \log \big( 1 + \frac{\mu \lambda}{(\mu-\lambda)^2 N} \big) \Big) \to W,
$$
in distribution, as $N \to \infty$.  As $(\mu-\lambda)^2 N \to \infty$, and (from (\ref{eq.ST}))
$\widehat t = \big(s(\widehat x) - s(X_N(0)/N)\big)/(\mu-\lambda)$, this implies that
\begin{equation} \label{eq:widehat}
(\mu-\lambda) (T_N - \widehat{t}) \to W.
\end{equation}
Thus the distribution of $T_N$ is concentrated in a window of width of order $1/(\mu-\lambda)$ around the guide time~$\widehat{t}$.

We now explain briefly how to express the probability of extinction by time~$t$, asymptotically, in terms of the deterministic process $x(t)$.  
For $w$ a fixed constant, set $t_w = \widehat{t} + w/(\mu-\lambda)$.
From~(\ref{eq:widehat}), we have that $\Pr (T_N \le t_w) = e^{-e^{-w}} + o(1)$.
From (\ref{eq.ST}), we see that, for any $\alpha$, and any constant $w$,
$$
t_\alpha(\widehat{x}e^{-w}) - t_\alpha(\widehat{x}) = \frac{s(\widehat{x}e^{-w}) - s(\widehat{x})}{\mu-\lambda} = \frac{w+o(1)}{\mu-\lambda},
$$
noting that $\widehat{x}/(\mu-\lambda) = o(1)$.  It follows that, for any value of $\alpha = x(0)$, and any fixed $w$,
$$
x(t_w) = \widehat{x} e^{-w} (1+o(1)),
$$
and therefore
$$
\exp \big( - N x(t_w) (\mu-\lambda)/\mu \big) = e^{-e^{-w}} + o(1) = \Pr(T_N \le t_w) + o(1).
$$
One can now see that
\begin{equation} \label{eq.relationship}
\sup_t \Big| \Pr(T_N \le t) - \exp \big( - N x(t) (\mu-\lambda)/\mu \big) \Big| \to 0 \quad \mbox{ as } N \to \infty.
\end{equation}

By the time $\widetilde{t} \simeq \widehat{t}$ when the number $X_N(\widetilde{t})$ of infectives has dropped to $\mu/(\mu-\lambda)$, the epidemic 
is well within its final phase, and, as we have shown, after time $\widetilde t$ it is well-approximated by a linear birth-and-death chain, or by 
a subcritical branching process.  The behaviour of such a process is well-understood: once it reaches a level of order $1/(\mu-\lambda)$, it 
fluctuates through states of that order until it goes extinct.

This is also an illustration of the effect of parameter choice on cut-off.  Away from criticality, we have a strong cut-off phenomenon: the extinction
time is concentrated within a window of time much shorter than the overall extinction time, reflecting the idea that, before the window, the process is ``large'' with high probability, and it is unlikely to drop to~0 very quickly.  As we approach criticality, once the process drops to order $\mu/(\mu-\lambda)$, it can (but does not always) stay around that level for a relatively long time, giving a weaker cut-off.

As set out in the Introduction, we expect these findings to extend to a wide class of epidemic models (not only SIS models).
In their final stages, many models will be well-approximated by a barely subcritical branching process independent of population size,
and the nature of this approximating branching process will govern the final stages of the epidemic, for suitable parameter values.  So we expect the prevalence curve of an epidemic to follow the solution of a differential equation closely until the
number of infectives becomes small, and the actual time of extinction to fall in a window of time of width of order $1/\mu(1-\cR_0)$, containing
within it the point where the differential equation predicts the number of infectives to be $1/(1-\cR_0)$.  We expect a strong cut-off for epidemics
away from criticality -- once the process approaches extinction, it goes extinct very quickly -- and a weaker cut-off as we approach criticality.

A typical sample path for a barely subcritical epidemic will resemble the sample paths of the SIS logistic process in Figure~\ref{fig:mc}: they
reach states of order $1/(1-\cR_0) = \mu/(\mu-\lambda)$ following a fairly smooth trajectory, but then fluctuate around that level for a period
of time of order $1/\mu(1-\cR_0) = 1/(\mu-\lambda)$, possibly nearing extinction several times, until finally the epidemic does die out.

To illustrate (\ref{eq.relationship}), we simulated the relationship between the mean prevalence of infection $\mathbb{E} X_N(t)/N$ (as computed 
from the Kolmogorov forward equations), which is very close to the deterministic process $x(t)$), and the probability of extinction 
$\mathbb{P}(X_N(t)=0)$ (as computed from Theorem~\ref{thm.main}).  We used the parameter values $N=X_0 = 10^7$, $\mu = 1$ and
a variety of different values of $\mathcal{R}_0 = \lambda$: the results are shown in Figure~\ref{fig:ItQt}.  

As a first step towards more realistic models, there has been recent interest in a variant of an SIS epidemic where
the durations of each case of infection are iid random variables $Q_i$ with mean~1, not necessarily having an exponential distribution.  Ball, Britton
and Neal~(2016) show that, if the epidemic starts with a single individual, the expected duration of the epidemic does not depend on the
distribution of the $Q_i$.  It would be interesting to investigate whether our results can be extended to this more general setting.

\clearpage

\begin{figure}[h]
\centering
\includegraphics[width = 0.8\textwidth]{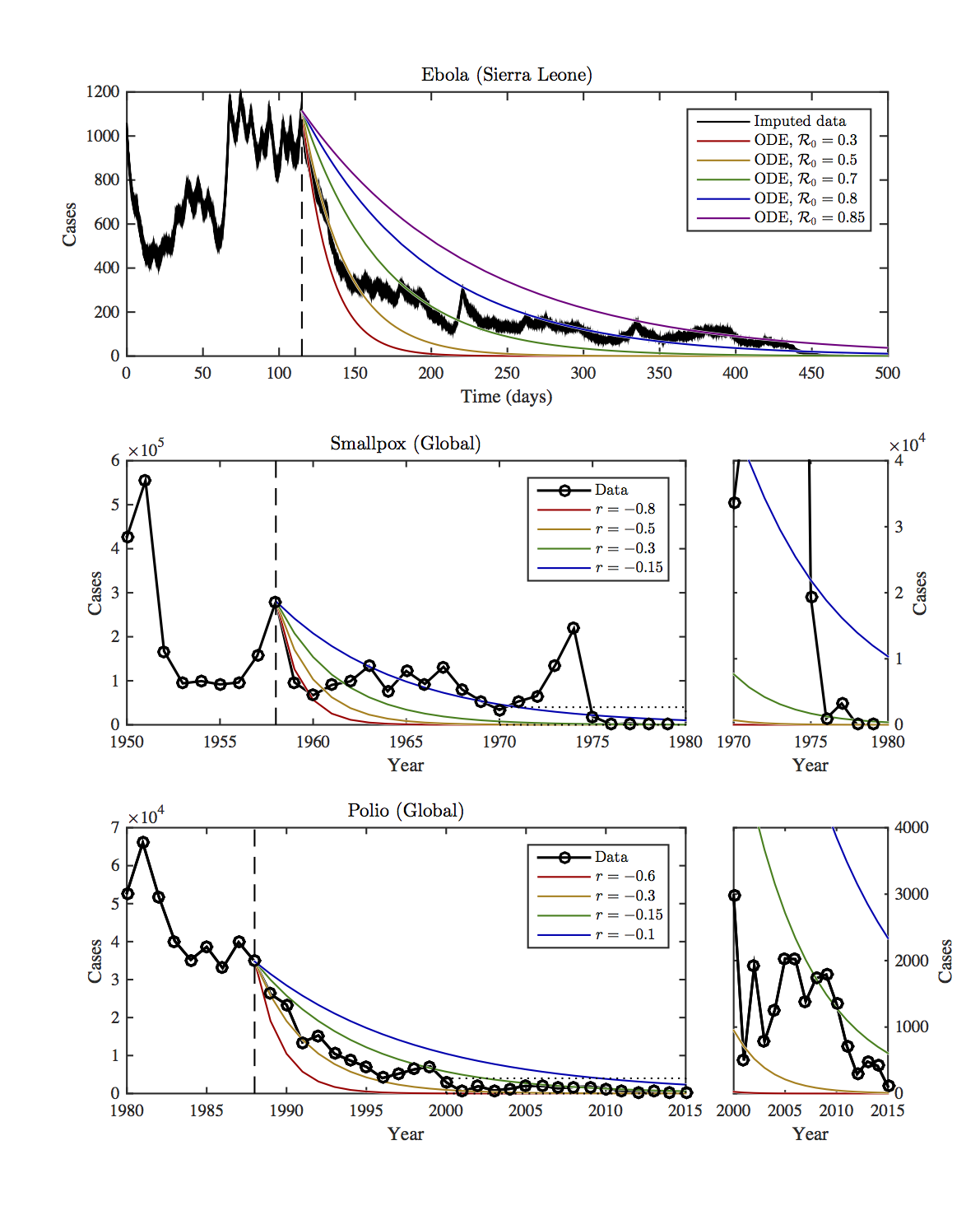}
	\caption{Data relating to diseases placed under control. For smallpox and
		polio the years when official eradication efforts started are indicated
		with dashed vertical lines. For Ebola, the effective start of control is
		estimated by eye and indicated by a dashed vertical line. Exponential decay
		curves with different rates are superimposed on the later part of the data.
        See Appendix for more details.
}
\label{fig:data}
\end{figure}

\clearpage

\begin{figure}[h]
\centering
\includegraphics[width = 0.99\textwidth]{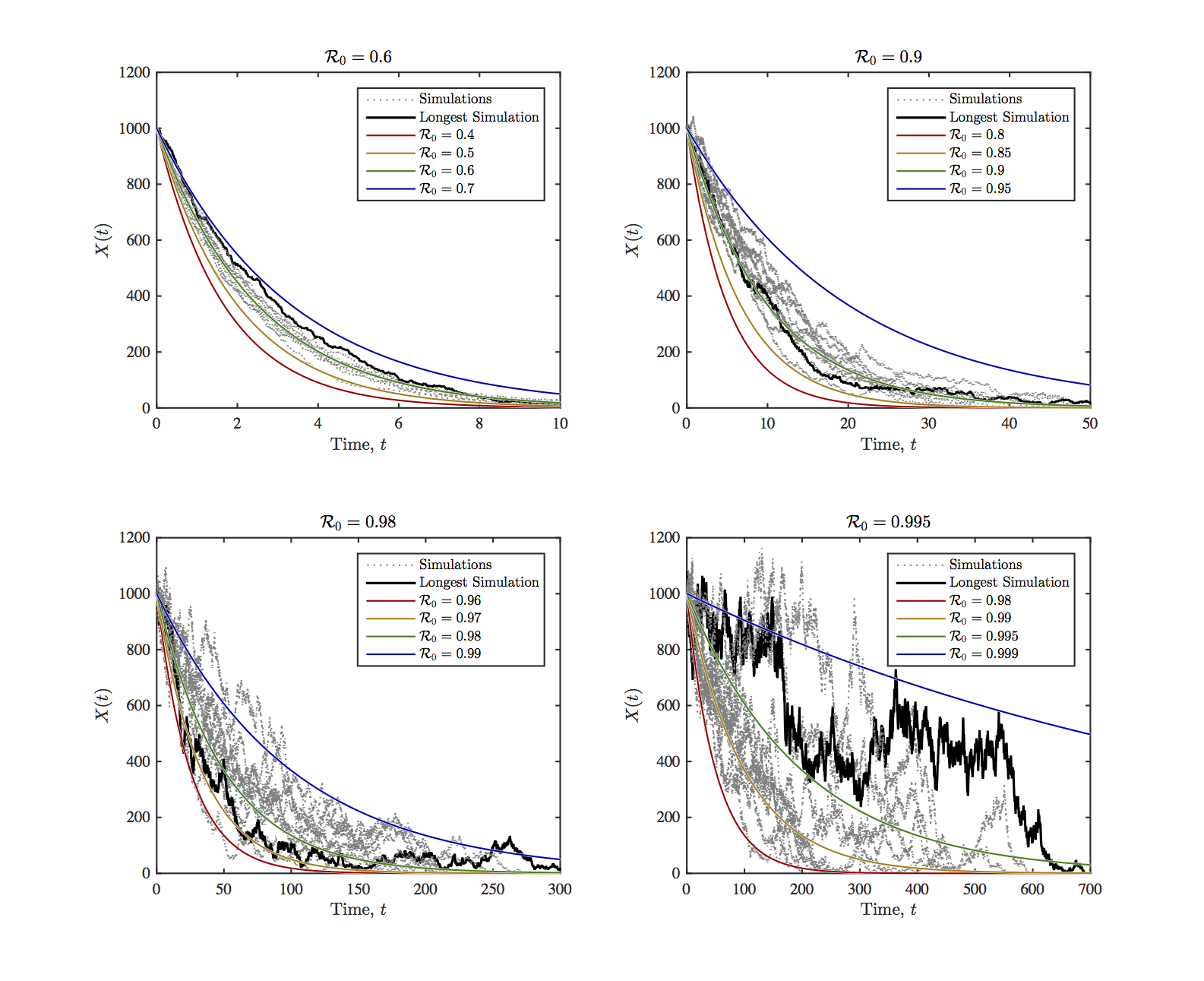}
	\caption{The SIS logistic model at different values of $\cR_0$. Other parameter
		choices are $\mu = 1$, $N = 10^6$ and $X_N(0) = 10^3$. Exponential decay curves
		with different rates are superimposed on ten realisations, with the longest one
	emphasised.
}
\label{fig:mc}
\end{figure}

\clearpage

\begin{figure}[h]
\centering
\includegraphics[width = 0.99\textwidth]{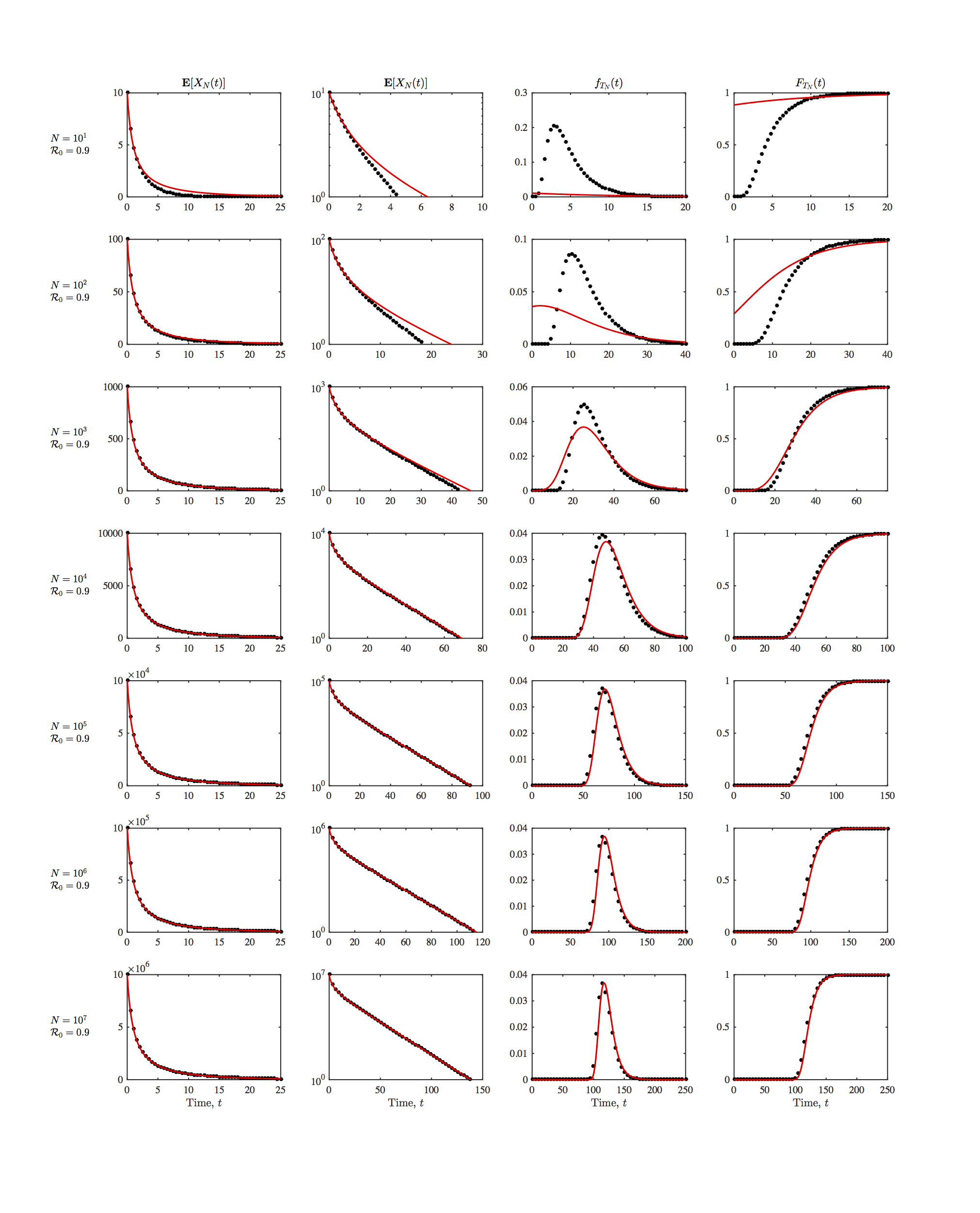}
	\caption{Comparison of simulations of the Kolmogorov forward equations (black dots) with
	asymptotic results (red solid lines). First column: Expected number of
	infectives with a linear $y$-axis.  Second column: Expected number of
	infectives with a logarithmic $y$-axis. Third column: probability density
function for the extinction time. Fourth column: distribution function for the
extinction time. Rows represent different values of $N$ from $10^1$ to $10^7$
with $\mathcal{R}_0$ unscaled.
}
\label{fig:unscaled}
\end{figure}

\begin{figure}[h!]
\centering
\includegraphics[width = 0.99\textwidth]{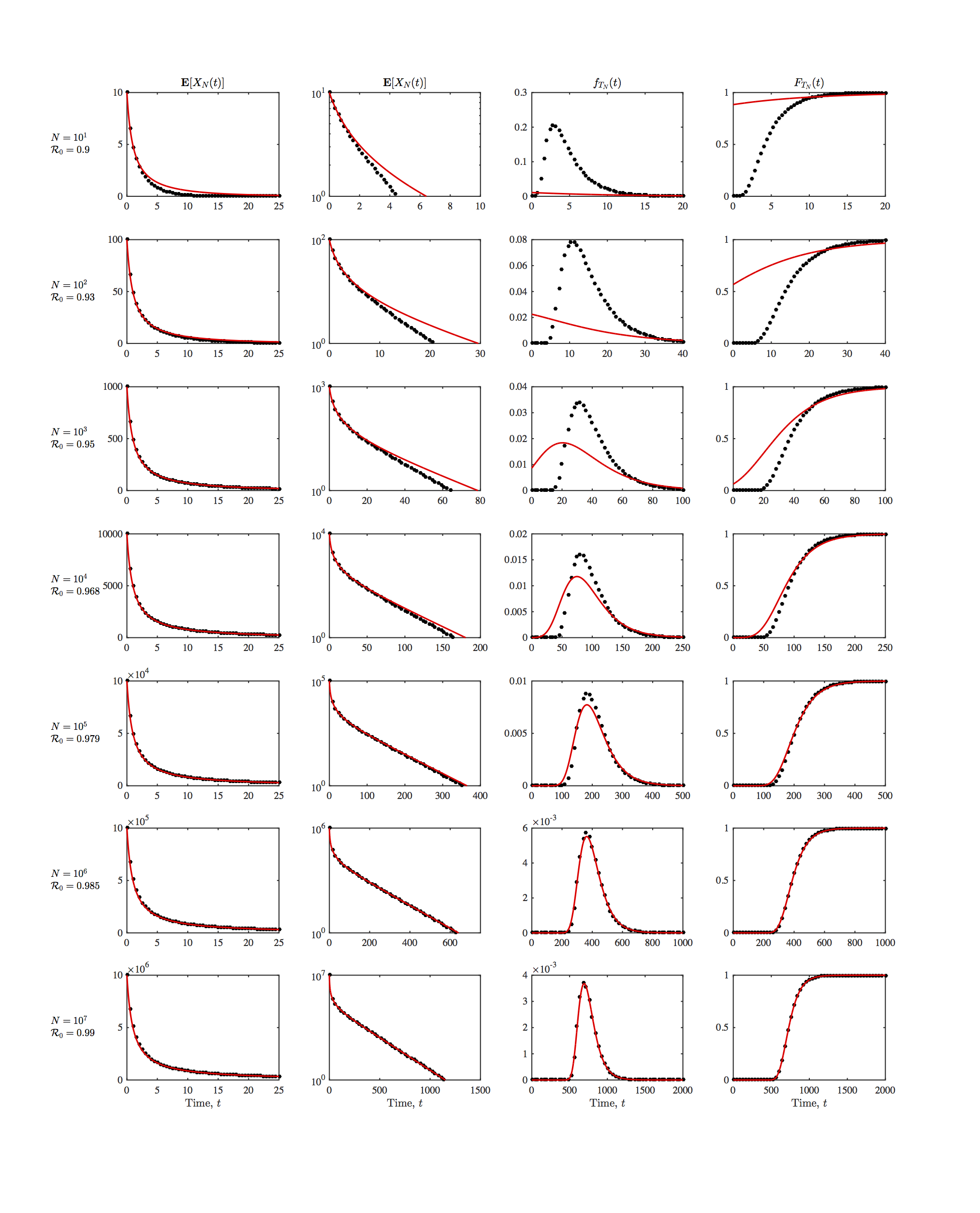}
	\caption{
	Comparison of simulations of the Kolmogorov forward equations (black dots) with
	asymptotic results (red solid lines). First column: Expected number of
	infectives with a linear $y$-axis.  Second column: Expected number of
	infectives with a logarithmic $y$-axis. Third column: probability density
function for the extinction time. Fourth column: distribution function for the
extinction time. Rows represent different values of $N$ from $10^1$ to $10^7$
with $1-\mathcal{R}_0$ scaling approximately like $N^{-1/3}$.
}
\label{fig:scaled}
\end{figure}

\begin{figure}[h!]
\centering
\includegraphics[width = 0.99\textwidth]{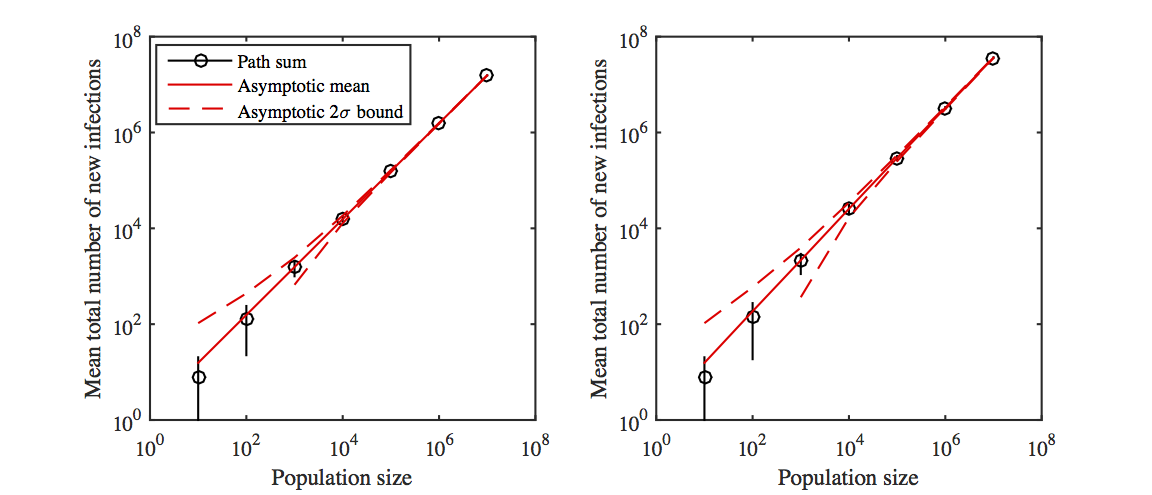}
	\caption{
		Comparison of total number of infections $C_N$ with $N$	for
		(left) the parameter choices in Figure~\ref{fig:unscaled} above and (right)
		the parameter choices in Figure~\ref{fig:scaled} above. Calculations for
		the path sum are shown as black circles, with vertical black lines showing
		$\pm 2$ standard deviations.  The asymptotic formula for large $N$ is shown
		as red lines, with the asymptotic bounds on $\pm 2$ standard deviations shown
		as red dashed lines.
}
\label{fig:cns}
\end{figure}

\begin{figure}[h!]
\centering
\includegraphics[width = 0.99\textwidth]{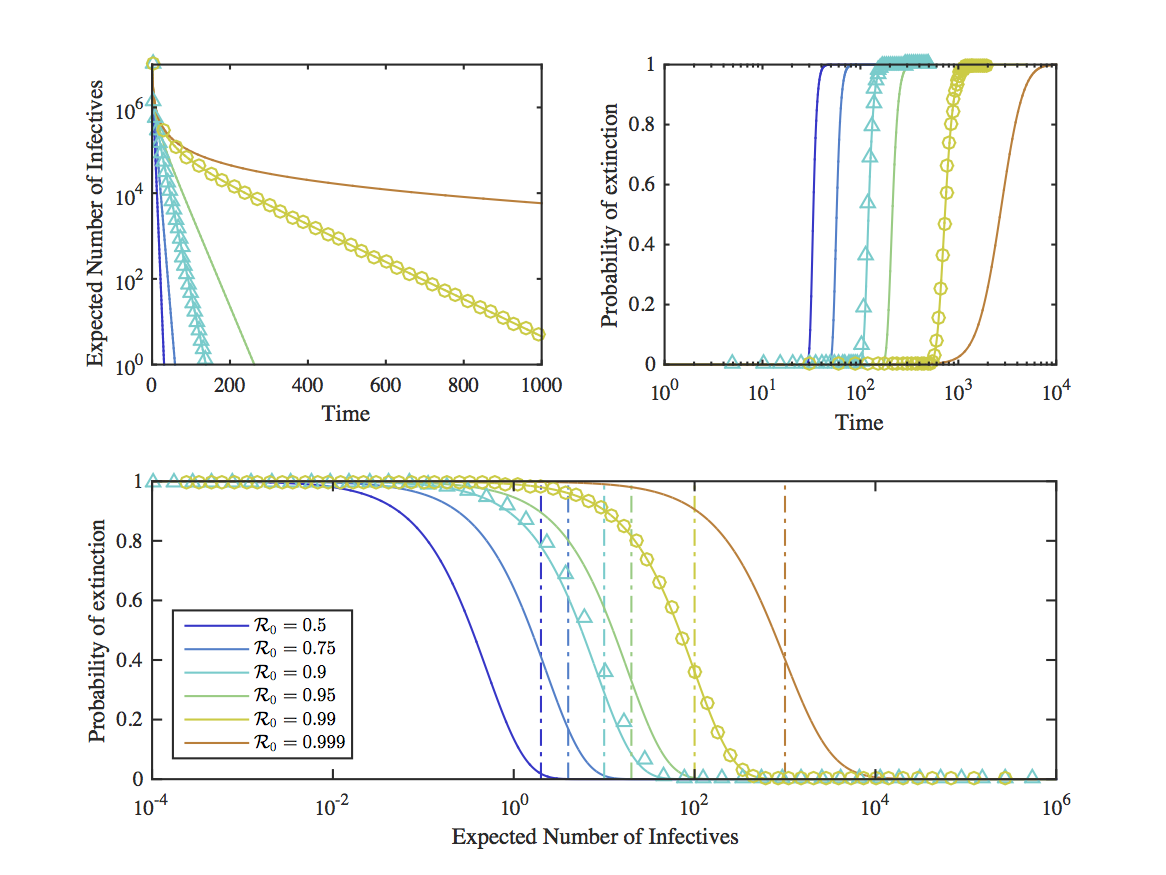}
	\caption{
	Comparison of asymptotic mean number of infectives $\E X_N(t)$ with
	asymptotic probability of extinction $\mathbb{P}(X_N(t)=0)$ (from Theorem~\ref{thm.main}) for
	$N=X_N(0) = 10^7$, at a variety of different values of $\mathcal{R}_0$ for $\mu =
	1$. The non-asymptotic results due to numerical integration of the Kolmogorov forward
	equations from Figures~\ref{fig:unscaled} and~\ref{fig:scaled} are shown as
	markers next to the appropriate asymptotic curves.  The vertical dot-dash lines in the bottom figure show the
    point where $N x = \mu/(\mu-\lambda)$.
}
\label{fig:ItQt}
\end{figure}

\clearpage

\appendix \label{sec:ebola}

\section{Numerical evidence from real epidemics}

We present here data from three real epidemics, pictured in Figure~\ref{fig:data}, where the behaviour near extinction fits our description of a barely
subcritical epidemic.  We make no claim that the real epidemics are well-modelled by any particular stochastic process.  In reality, the available data will
give only a partial picture of the true spread of disease, and the parameter values of any process will vary widely in time and geographical location.

Our first example is a simulation of the Ebola epidemic in 2014/15, based on the real-time study on Ebola in Sierra Leone
performed by Camacho et al~(2015).

The model is a stochastic compartmental model, with individuals in six different classes: $S$ (susceptible), $E_1$ and $E_2$ (two non-infectious latent
classes), $I_c$ (infectious cases in the community), $I_h$ (hospitalised infectious cases) and $R$ (removed).  The transitions can be
encoded as:
\begin{eqnarray*}
(S, E_1) &\to& (S-1,E_1+1) \quad \mbox{ at rate } \quad \beta (I_c + I_h), \\
(E_1, E_2) &\to& (E_1-1,E_2+1) \quad \mbox{ at rate } \quad 2 \nu E_1, \\
(E_2, I_c) &\to& (E_2-1,I_c+1) \quad \mbox{ at rate } \quad 2 \nu E_2, \\
(I_c, I_h) &\to& (I_c-1,I_h+1) \quad \mbox{ at rate }\quad \tau I_c, \\
(I_h, R) &\to& (I_h-1,R+1) \quad \mbox{ at rate } \quad \gamma I_h.
\end{eqnarray*}
For instance, the first line represents an individual moving from class $S$ to class $E_1$ at rate proportional to the number of infectives: we have
simplified the model for our purposes by assuming that the number of susceptibles is constant throughout, and that the infection rate $\beta$ is
constant throughout the epidemic.

The parameter values were fitted to data by Camacho et al~(2015) using a computationally intensive statistical framework in which
$\beta$ varies over time: the overlaid curves in Figure~\ref{fig:data} show that fixed $\beta$ --
and therefore fixed $\cR_0 = \beta(\gamma^{-1} + \tau^{-1})$ -- does not work.
The values derived from the data are $\nu^{-1} = 9.4$ days (so the latent period has a 2-Erlang distribution with mean $\nu^{-1}$),
$\tau^{-1} = 4.3$ days and $\gamma^{-1} = 6.9$ days.

For the simplified model, the population means in each compartment obey the ODEs
\ba
\ddt{\E X_{E_1}} & = \beta (\E X_{I_c} + X_{I_h}) - 2\nu \E X_{E_1} \text{ ,} &
\ddt{\E X_{E_2}} & = 2\nu(\E X_{E_1} - \E X_{E_2}) \text{ ,}\\
\ddt{\E X_{I_c}} & = 2\nu \E X_{E_2} - \tau \E X_{I_c} \text{ ,} &
\ddt{\E X_{I_h}} & = \tau \E X_{I_c} - \gamma \E X_{I_h} \text{ .} \label{ebodes}
\ea

Based on this model, we generated estimates for the total number of infectives of Ebola over time.
We took publicly available data on cumulative incidence for Ebola in Sierra
Leone, and assumed: (i)~that individual cases moved from $I_c$ to $I_h$ at some time $t$
with uniform probability density in the day before they are reported; (ii)~that
the latent period ended at a time $t-t_1$, where $t_1$ is exponentially distributed with
rate~$\tau$; (iii)~that the infectious period ended at a time $t + t_2$, where $t_2$ is
exponentially distributed with rate~$\gamma$.

Simulating from this process ten times gives the black lines for the number of cases over time
in Figure~\ref{fig:data} (top plot), which are also compared to numerical solutions to~\eqref{ebodes} for different values of $\cR_0$,
shown as coloured lines.


Our other two examples are broad-brush pictures of the courses of well-known epidemics, where we present simply the number of recorded cases in each year.
These examples are smallpox (Figure~\ref{fig:data}, middle plot), which
was subject to a successful global eradication campaign, and polio
(Figure~\ref{fig:data}, bottom plot), which is currently subject to a global
eradication campaign that will hopefully be successful soon.  In these cases, we do
not compare to specific ODE models, but instead to exponential decay curves of
the form $e^{-r t}$, for various values of $r$, which are shown as coloured lines.

For all three real examples, we see that the duration of the epidemics after they have been brought under control is longer than might be predicted 
from the smooth curves, which would be associated with straightforwardly subcritical epidemics with much smaller extinction times, of order $\log N/\mu$.  
We also see that the prevalence does not decay smoothly in time in its final stages, but goes up and down several times before extinction.

We take the smallpox data as an example, to show how the data might be seen to fit our predicted behaviour for a barely subcritical process in a rough
quantitative sense.
The highly infectious period for smallpox is on the order of a week, which corresponds to a value of $\mu^{-1}$
around 0.02 years.  The middle graph in Figure~1 shows the total number of cases in each year, which is consistent with the number of cases remaining of
order about 2000 over the period 1958-1973.  Aggregating cases over full years obscures any erratic behaviour within each year.  It seems clear that
the effective value of $\cR_0$ for the smallpox epidemic must have varied considerably over that period, due to seasonal effects, changes in control
measures both globally and in response to local outbreaks, and other factors.  However, the total number of cases of smallpox over the 15-year period
is around two million, and the total number of cases at any time in 1973 is within a few thousand of the number in 1958, indicating that the average
value of $\cR_0$ over this period is within about 0.001 of~1.  For a postulated average value of $\cR_0$ of 0.999, our hypothesis
would suggest a period of time of order 20 years before extinction, during which the number of cases would resemble a random walk remaining of order
at most 1000, which is a reasonable fit to the data shown.

\end{document}